\documentclass[arma]{svjour}
\usepackage{mathrsfs,amssymb,amssymb,amsmath}
\usepackage{array}

\newcommand{\ud}{\mathrm{d}}
\newcommand{\e}{\mathrm{e}}

\newcommand{\R}{\mathbb{R}}
\newcommand{\N}{\mathbb{N}}
\newcommand{\E}{\mathcal{E}}
\newcommand{\cO}{\mathcal{O}}
\newcommand{\OP}{\mathbb{P}}
\newcommand{\bn}{{\bf n}}

\newcommand{\cA}{\mathcal{A}}
\newcommand{\cdiv}{\nabla\cdot}
\newcommand{\lt}{\left}
\newcommand{\rt}{\right}
\newcommand{\paro}{\partial\Omega}
\newcommand{\grad}{\nabla}

\begin{document}

\title{The motion of a fluid-rigid disc system at the
zero limit of the rigid disc radius}
\author{ Masoumeh Dashti \and James C. Robinson}
\titlerunning{Limiting motion of a fluid-rigid body system}
%
%
\maketitle
\begin{abstract}
We consider the two-dimensional motion of the coupled system of 
a viscous incompressible fluid and a rigid disc moving with the fluid,
in the whole plane. The fluid motion is described by the Navier-Stokes equations
and the motion of the rigid body by conservation laws of linear and angular momentum.
We show that,
assuming that the rigid disc is not allowed to rotate, 
as the 
radius of the disc goes to zero, the solution of this 
system converges, in an appropriate sense, to the solution of 
the Navier-Stokes equations describing the motion of only fluid 
in the whole plane. We also prove that the trajectory of the 
centre of the disc, at the zero limit of its radius, coincides 
with a fluid particle trajectory.

\end{abstract}


\section{Introduction}
The use of rigid tracers for finding the Lagrangian paths of the 
fluid flow is based on the assumption that at the zero limit of 
the rigid body radius, the trajectory of the rigid tracer coincides 
with a fluid particle trajectory. 
Studying the validity of this assumption has been the motivation for the problem that 
we address in this paper.
We consider the system of one rigid disc moving with the fluid flow in $\R^2$,
assume the disc does not rotate around its centre,
and show the convergence of the trajectory of the centre of the disc to a fluid
particle trajectory.

We consider the two-dimensional domain $\R^2$ occupied by an incompressible fluid of density $1$ and viscosity $\nu$ and a rigid disc of radius $r$ and density $\rho$, which here we set it equal to the density of the fluid. We assume that the motion of the fluid is modelled by the Navier-Stokes equations
with no-slip condition on the boundary of the rigid disc and consider the rigid body motion to be described
by the equations of the balance of linear and angular momentum: 
\begin{align}
\frac{\partial U}{\partial t} - \nu\Delta U+(U\cdot\nabla)U+\nabla P=F
                         &,\quad x\in\R^2\setminus B(t)\label{fluid-newton0}\\
\nabla\cdot U=0&,\quad x\in\R^2\setminus B(t)\label{div-free0}\\
U=\dot{h}(t)+\omega(t)(x-h(t))^{\bot}
                         &,\quad x\in \partial B(t)\label{fluid-bdry0}\\
M\ddot{h}(t)=-\int_{\partial B(t)}\Sigma \bn\,\ud \Gamma+M\xi&,\label{rigid-newton0}\\
J\dot\omega(t)=-\int_{\partial B(t)} (x-h(t))^{\bot}\cdot\Sigma \bn\,\ud\Gamma+J\eta&,\label{rigid-momentum0}\\
U(x,0)=U_0(x)&,\quad x\in\mathbb{R}^2\setminus B(0),\label{fluid-ic0}\\
h(0)=(0,0),\quad \dot{h}(0)=\dot{h}_0\in\R^2, \quad \omega(0)=\omega_0\in\R\label{rigid-ic0}.
\end{align}
In the above system $t\in[0,T]$, $B(t)$ is the region occupied by the disc at time $t$, $U$ is the velocity vector of the fluid, $P$ the pressure scalar field, $h(t)$ the position of the centre of the rigid disc at time $t$, $\omega$ its angular velocity and 
$M=\rho\pi r^2$ and $J=Mr^2/2$ are its mass and moment of inertia respectively. The stress tensor $\Sigma$ is defined as
$$
\Sigma=-P\mathrm{Id}+2\nu D(U),
$$
where
$$
(D(U))_{k,l}=\frac{1}{2}\big(\frac{\partial U_k}{\partial x_l}+\frac{\partial U_l}{\partial x_k}\big).
$$
The body force applied to the fluid is denoted by $F$, and $M\xi$ and $J\eta$ are the force and torque applied to the disc.
The solvability of the above system, and also the equivalent system in a three-dimensional domain, is known.
When the domain of the motion
is the whole of $\R^3$, this is shown by Judakov (1974) and Serre (1987) using Galerkin approximations. 
In the case of a bounded domain, the global existence of 
at least one weak solution is proved 
by Hoffmann \& Starovoitov (1999) for the two-dimensional case 
and by Conca, San Martin \& Tucsnak (2000) and Gunzburger, Lee \& Seregin (2000) for three-dimensional domains
under some constraints on collision of the rigid body and the boundary of the domain. 
The method of Hoffmann \& Starovoitov (1999) is based on approximating the rigid body by a very viscous fluid,
Conca et al (2000) write the system in a coordinate system attached to the body and use a Galerkin method, 
and Gunzburger et al (2000) approximate the system by time-discretized problems.

In their 2004 paper, Takahashi and Tucsnak show that the system of 
a fluid-rigid disc in the whole plane has a unique global strong solution using a contraction mapping approach. 
This result is shown to be true for the case of a bounded two-dimensional domain as well by Takahashi (2003). 
When the motion is in $\R^3$, the existence of a local unique strong solution is proved by Galdi and Silvestre (2002).

In the case of the motion of several rigid bodies in a fluid, 
at least one weak solution is known to exist 
(Desjardins \& Esteban 1999; Desjardins \& Esteban 2000; Feireisl 2002; Grandmont \& Maday 2000; San Martin, Starovoitov \& Tucsnak 2002). 
The uniqueness however is not known.

We also note the papers by Robinson (2004) and Iftimie, Lopes Filho \& Nussenzveig Lopes (2006). In Robinson (2004) a simplified model of the motion of a fluid-particle system
at the limit of zero radius of particles is studied. The simplification is in considering only
momentum exchange between fluid and particle and assuming that the domain of the fluid
does not change in time; and it is shown that in the limit the motion of the fluid is 
described by the standard Navier-Stokes equations. 
Iftimie et al (2006) consider the two-dimensional 
asymptotic motion of a fluid outside a small obstacle as the size of the obstacle tends to zero.
They prove that the limit flow satisfies the Navier-Stokes equations in the full plane.

Knowing that a unique strong solution of (\ref{fluid-newton0})--(\ref{rigid-ic0}) exists for any $r>0$ 
by the result of Takahashi and Tucsnak (2004), in this paper we study the limiting behaviour of the
solution of (\ref{fluid-newton0})--(\ref{rigid-ic0}) as $r\to 0$. 
To do this, we 
do not take into account the torque exerted by the fluid on the particle,
i.e.\ we do not allow the particle to rotate.
Of course, one could incorporate this into a model
with angular momentum by considering $\omega_0=0$ and imposing a torque
$$
J\eta=\int_{\partial B(t)} (x-h(t))^{\bot}\cdot\Sigma \bn\,\ud\Gamma
$$ 
so that
$\dot\omega=0$, but this simply has the effect of removing equation
(\ref{rigid-momentum0}) from all our calculations.
It is of course desirable to study the more general system when the disc can rotate,
but with a nonzero $\omega$ we have not been able to obtain the required uniform 
estimates 
(see Remark \ref{r:omega} in Section \ref{sec:uniform}).
In order to prove the convergence of the velocity field as $r\to 0$ one requires uniform bounds on the fluid velocity field for $r>0$ that are not provided by the method of Takahashi and Tucsnak. Here, we obtain such uniform bounds over a small enough time interval $[0,T]$.
We then show that at the zero limit of $r$, the velocity field 
satisfies the Navier-Stokes equations in the whole plane. 
Then using the properties of the strong solutions of the 
Navier-Stokes equations, we can prove this convergence for all times.

In what follows, in Section \ref{sec:prelim} we recall the result of Takahashi \& Tucsnak (2004) on the existence of strong solutions, and also their functional setting for the problem, which we will follow in this paper. 
In Section \ref{sec:Pr} we prove some of the properties of the functional spaces introduced in Section \ref{sec:prelim}.
In Section \ref{sec:uniform} we show that the solution of the above fluid-rigid system is bounded independent of the radius of the rigid disc, in the time interval $[0,T]$ for a small enough $T$ which is also independent of $r$.
In Section \ref{sec:limiteqn} we show that as $r\to 0$, the above 
fluid rigid system converges to the Navier-Stokes equations in 
the whole of $\R^2$ for all times and the trajectory of the 
centre of the disc converges to a fluid particle trajectory.

\section{Preliminaries}\label{sec:prelim}

In this section, following Serre (1987), Conca et al (2000) and Takahashi and Tucsnak (2004), we
write (\ref{fluid-newton0})--(\ref{rigid-ic0}) with respect to a coordinate system moving with the disc centre which results in the equations of motion of the fluid and rigid body in domains which are fixed with respect to time. 
We will also state and consider the functional setting introduced in Takahashi and Tucsnak (2004) for this problem.

Using the following change of variables
\begin{align*}
y(t)&=x(t)-h(t),\\
u(y,t)&=U(y+h(t),t),\quad p(y,t)=P(y+h(t),t),\\
\sigma(y,t)&=-p(y,t)\mathrm{Id}+2\nu D(u)(y,t)
\end{align*}
the equations of motion can be written as 
\begin{align}
\frac{\partial u}{\partial t} - \nu\Delta u+(u\cdot\nabla)u-(\dot{h}(t)\cdot\nabla)u+\nabla p=0
                         &,\quad y\in\Omega_r\label{fluid-newton}\\
\nabla\cdot u=0&,\quad y\in\Omega_r\label{div-free}\\
u=\dot{h}(t)+\omega(t)y^{\bot}
                         &,\quad y\in\paro_r\label{fluid-bdry}\\
M\ddot{h}(t)=-\int_{\partial B_r}\sigma \bn\,\ud \Gamma&,\label{rigid-newton}\\
J\dot\omega(t)=-\int_{\partial B_r} y^{\bot}\cdot\sigma \bn\,\ud\Gamma
                          +J\eta&,\label{rigid-momentum}\\
u(y,0)=u_0(y)&,\quad y\in\Omega_r,\label{fluid-ic}\\
h(0)=(0,0),\quad \dot{h}(0)=\dot{h}_0\in\R^2, \quad \omega(0)=\omega_0\in\R\label{rigid-ic}.
\end{align}
with $B_r=B(0)$ and $\Omega_r=\R^2\setminus B_r$.\\

\indent We define
\begin{align}
H_r&=\{u\in[L^2(\R^2)]^2:\nabla\cdot u=0 \quad\mbox{in }\R^2,
              \quad D(u)=0\quad\mbox{in } B_r\}\label{H_r}\\
V_r&=\{u\in[H^1(\R^2)]^2:\nabla\cdot u=0 \quad\mbox{in }\R^2,
              \quad D(u)=0\quad\mbox{in } B_r\}\label{V_r}
\end{align}
and consider the inner product on $H_r$ to be
$$
(u,v)=\int_{\Omega_r}(u\cdot v)\,\ud y+\rho\int_{B_r}(u\cdot v)\,\ud y
$$
which is equivalent to the standard inner product on $[L^2(\R^2)]^2$. 
We denote the orthogonal projection of $[L^2(\R^2)]^2$ onto $H_r$ by $\OP_r$.

\indent We let
\begin{align}
D(A_r)=\Big\{u\in[H^1(\R^2)]^2:\,
              &u\in[H^2(\Omega_r)]^2,\nonumber\\
              &\nabla\cdot u=0 \quad\mbox{in }\R^2,
              \quad D(u)=0\quad\mbox{in } B_r\Big\}\label{D(A_r)}
\end{align}
and for any $u\in D(A_r)$ define
\begin{align*}
\cA_r u=\left\{ 
  \begin{array}{l}
     -\nu\Delta u\quad\mbox{in }\Omega_r\\
     \frac{2\nu}{M}\int_{\partial B}D(u)\bn\,\ud y
         +\big[\frac{2\nu}{J}\int_{\partial B_r}y^{\bot}\cdot D(u)\bn\,\ud y\big]y^{\bot}
                                  \quad\mbox{in }B_r
  \end{array}\right.
\end{align*}
and 
\begin{equation}\label{eq:Ar}
A_r u=\OP_r \mathcal{A}_r u.
\end{equation}
It can be shown (Takahashi \& Tucsnak 2004) that $A_r$ is self adjoint:
$$
(A_r u,v)_{L^2(\R^2)}=2\nu\int_\Omega D(u):D(v)\,\ud y=(u,A_r v)_{L^2(\R^2)}.
$$
We can extend $u\in L^2(\Omega_r)$ to a function in $H_r$, defined on the whole of $\R^2$, by letting
$$
u(y)=\dot{h}(t)+\omega y^{\bot}, \quad \mbox{in } B_r.  
$$
We define
$$
\beta(u-\dot{h},u)=\lt\{
\begin{array}{ll}
\lt( (u-\dot{h})\cdot\nabla \rt)u & \mbox{in } \Omega_r\\
0 & \mbox{in } B_r.
\end{array}
\rt.
$$
Then using $A_r$ and $\OP_r$ defined as above, equations (\ref{fluid-newton})--(\ref{rigid-ic}) can be written as
\begin{align}\label{eq:functional}
\frac{\ud u}{\ud t}+A_r u=
-\OP_r\,\beta(u-\dot{h},u).
\end{align}

Existence of a strong solution of (\ref{fluid-newton})-(\ref{rigid-ic}) is shown by Takahashi \& Tucsnak (2004). 
We recall their result:


\begin{theorem}\label{thm:taktuc} (Takahashi \& Tucsnak 2004)
{\bf i)} Suppose $u_0\in H^1(\Omega_r)$ with $\cdiv u_0=0$ in $\Omega_r$ and $u_0(y)=\dot{h}_0+\omega_0 y^{\bot}$ on $\partial\Omega_r$. 
Then, there exists a unique strong solution of (\ref{fluid-newton})--(\ref{rigid-ic}), satisfying
\begin{align*}
&u\in L^2(0,T;H^2(\Omega_r))\cap C(0,T;H^1(\Omega_r))\cap H^1(0,T;L^2(\Omega_r)),\\
&p\in L^2(0,T;H^1),\quad h\in H^2(0,T;\R^2),\quad \omega\in H^1(0,T;\R),
\end{align*}
for any $T>0$.

{\bf ii)} Let $u_0\in L^2(\Omega_r)$ with $\cdiv u_0=0$. Then the system (\ref{fluid-newton})--(\ref{rigid-ic}) has a unique weak solution with $u\in L^2(0,T;H^1(\Omega_r))\cap C(0,T;L^2(\Omega_r))$.
 
\end{theorem}

The estimates in the result of Takahashi and Tucsnak (2004) depend on $r$ via
$\|\dot{h}\|_{L^\infty(0,T;\R^2)}$ (they use $M|\dot h|^2\le c(1+\|u_0\|_{H^1})$) and the constant $c$ in the inequality $\|\nabla^2 u\|_{L^2(\Omega)}\le c\,\|A_r u_r\|_{L^2}$ used in their proof. 
But here, to study the solution when $r\to 0$, we need uniform bounds on the velocity field.
In Section \ref{sec:uniform} we show that if $\omega(t)=0$ for
$t\ge 0$, one can find such uniform bounds on a time interval $[0,T]$, with $T$ only depending on the initial velocity field $u_0$. But first we prove some of the properties of the orthogonal projection $\OP_r$ which we will need later.


\section{Properties of the orthogonal projection $\OP_r$}\label{sec:Pr}

In this section we prove some results characterizing $\OP_r$, which we will need 
in the following sections. 
In the first lemma we identify the orthogonal complement of $H_r$.
Then we use it to show that the image of a sufficiently regular and divergence-free function under $\OP_r$ converges in $[L^2(\R^2)]^2$ to the function itself as $r\to 0$, 
which will be key in proving the convergence of the fluid-rigid body system (\ref{eq:functional}) 
to the Navier-Stokes equations in the whole plane in Section \ref{sec:limiteqn}.
We note that the results of this section are in the general case where
$\omega(t)$ can be nonzero. 


\begin{lemma}\label{lem:G1G2}

Let
$$
H_r=\{u\in[L^2(\R^2)]^2:\nabla\cdot u=0\;\;\,\mbox{in }\R^2,
              \quad D(u)=0\;\;\,\mbox{in } B_r\},
$$
\begin{align*}
G_1=\{u\in [L^2(\R^2)]^2: u=\nabla q_1,\, \nabla q_1\in [L^2(\R^2)]^2,\, q_1\in L^1_{loc}(\R^2)\},
\end{align*}
\begin{align*}
G_2=\{
&u\in [L^2(\R^2)]^2:\, 
    \cdiv u=0\mbox{ in }\R^2,\,\\
&\; u=\nabla q_2 \mbox{ in } \Omega_r \mbox{ with } \nabla q_2\in [L^2(\Omega_r)]^2,\, q_2\in L^1_{loc}(\Omega_r),\\
&\; u=\phi \mbox{ in } B_r \mbox{ with } \phi\in [L^2(B_r)]^2,\mbox{ and }\int_{B_r} \phi\cdot y^{\bot}\,\ud y=0\}. 
\end{align*}
Then $H_r$, $G_1$ and $G_2$ are mutually orthogonal and 
$$
[L^2(\R^2)]^2=H_r\oplus G_1\oplus G_2\,.
$$
\end{lemma}
\begin{proof} The proof is mostly based on the ideas in the proof of Lemma 4 of Conca et al (2000).
Since $H_r$ and $G_2$ are subsets of the set of square integrable divergence-free functions, 
$G_1$ is perpendicular to both of them
(noting that since we have assumed $\rho=1$, the inner product over the above spaces is the
standard $L^2$-inner product).
To show that $H_r\bot G_2$, we consider 
$$
w=\lt\{ \begin{array}{ll} \nabla q_2 &\mbox{in }\Omega_r\\
                         \phi &\mbox{in }B_r
              \end{array}\rt. \in G_2
$$
and $v\in H_r$ with 
$v(y)=V_v+\omega_v y^{\bot}$ for $y\in B_r$
where 
$V_v\in \R^2$ and $\omega_v\in\R$ are constant.
We have
\begin{align}\label{eq:v.w}
(v,w)&=\int_{\Omega_r}v\cdot\nabla q_2\,\ud y+\int_{B_r}(V_v+\omega_v y^{\bot})\cdot\phi\,\ud y\nonumber\\
     &=V_v\cdot\int_{\partial\Omega_r} q_2\,\bn\,\ud y+\, V_v\cdot\int_{B_r} \phi\,\ud y.
\end{align}
Since $\nabla y$ is the identity matrix and $\phi$ is divergence-free, we can write
$$
\int_{B_r}\phi\,\ud y
= \int_{B_r}\nabla y\,\phi\,\ud y
     =\int_{B_r}\lt( \begin{array}{c} \cdiv(\phi y_1)\\ \cdiv(\phi y_2) \end{array} \rt)\,\ud y
     =\int_{\partial\Omega_r}\phi\cdot (-\bn)y\,\ud y\nonumber
$$
Now since $w\in [L^2(\R^2)]^2$ is divergence-free, the normal components of $\phi$
and $\nabla q_2$ on the boundary, by Theorem 1.2 of Temam 1977, satisfy 
\begin{align}\label{eq:curve-normal}
\int_{\partial\Omega_r}(\phi\cdot\bn-\nabla q_2\cdot\bn)\,g(y)\,\ud y=0
\end{align}
for any $g\in H^{1/2}(\partial\Omega_r)$. 
Therefore, with $g(y)=y$ and since $\Delta q_2=0$, we have
\begin{align*}
\int_{\partial\Omega_r}\phi\cdot (-\bn)y\,\ud y\nonumber
     &=-\int_{\partial\Omega_r}\nabla q_2\cdot \bn\, y\,\ud y\\
     &=-\int_{\Omega_r}
                \lt( \begin{array}{c} \cdiv(\nabla q_2 y_1)\\ \cdiv(\nabla q_2 y_2) \end{array} \rt)\,\ud y\\
     &=-\int_{\Omega_r}\nabla y\,\nabla q_2\,\ud y\\
     &=-\int_{\Omega_r}\partial_j q_2\,\delta_{ij}\,\ud y
        =-\int_{\Omega_r}\partial_j (q_2\,\delta_{ij})\,\ud y\\
     &=-\int_{\paro_r}q_2\,\delta_{ij}\,\bn_j\,\ud y
        =-\int_{\paro_r}q_2\,\bn\,\ud y
\end{align*}
where $i,j=1,2$, $\partial_j=\partial/\partial y_j$ and
$$
\delta_{ij}=\lt\{ \begin{array}{ll} 1,&\mbox{ if } i= j,\\ 0,&\mbox{ otherwise}. \end{array} \rt.
$$
Therefore
\begin{align}\label{eq:iphi}
\int_{B_r}\phi\,\ud y=-\int_{\partial\Omega_r}q_2\,\bn\,\ud y
\end{align}
and by (\ref{eq:v.w}), $(v,w)=0$.

Now we consider an arbitrary $u\in [L^2(\R^2)]^2$ and show that there exist a unique $v\in H_r$, with
\begin{equation}\label{eq:v_B}
v(y)=V_v+\omega_v y^{\bot}\quad\mbox{for } y\in B_r
\end{equation}
where $V_v\in \R^2$ and $\omega_v\in\R$ are constants, and also unique 
$\nabla q_1\in G_1$ and $w\in G_2$ with $w=\nabla q_2$ in $\Omega_r$ and $w=\phi$ in $B_r$, such that
\begin{align*}
u=v+\nabla q_1+w=\left\{\begin{array}{ll} v+\nabla q_1+\nabla q_2&\mbox{ in }\Omega_r,\\
                                    V_v+\omega_v y^{\bot}+\nabla q_1+\phi &\mbox{ in }B_r.
                                       \end{array}  \right.
\end{align*}
For $q_1$ we solve
\begin{align}\label{eq:q1}
\lt\{
\begin{array}{l}
\Delta q_1(y)=\cdiv u(y),\quad y\in\R^2,\\
\nabla q_1\to 0,\quad\mbox{as } |y|\to\infty.
\end{array}\rt.
\end{align}
The above system has a solution $q_1\in L^1_{loc}(\R^2)$ with $\nabla q_1\in L^2(\R^2)$, 
which is unique up to a constant. Clearly, $\nabla q_1\in G_1$.

We set $\phi:B_r\to [L^2(B_r)]^2$ to be
\begin{equation}\label{eq:phi}
\phi=u-\nabla q_1-V_v-\omega_v y^{\bot},
\end{equation}
with $V_v$ and $\omega_v$ as in (\ref{eq:v_B}). Integrating the above equation over $B_r$ and using (\ref{eq:iphi}) we can find $V_v$ in terms of $q_2$:
\begin{align}\label{eq:Vv}
V_v=\frac{1}{\pi r^2}\lt( \int_{B_r} (u-\nabla q_1)\,\ud y+\int_{\partial\Omega_r} q_2\,\bn\ud y \rt).
\end{align}
Now, for $q_2$ we write
\begin{align}\label{eq:q2}
\lt\{
\begin{array}{l}
\Delta q_2=0,\quad\mbox{in }\Omega_r,\\
\nabla q_2\cdot\bn=(u-\nabla q_1)\cdot\bn-V_v\cdot\bn,
                                        \quad\mbox{on }\partial\Omega_r.
\end{array}\rt.
\end{align}
Let 
$$
D^{1,2}=\{q\in L^1_{loc}(\Omega_r): \nabla q\in [L^2(\Omega_r)]^2\},
$$
with the seminorm $\|q\|_{D^{1,2}}=\|\nabla q\|_{L^2(\Omega_r)}$.
For any $\xi\in D^{1,2}$, we have
\begin{align*}
0=(\Delta q_2,\xi)
&=(\nabla q_2,\nabla \xi)+\int_{\partial\Omega_r}\xi\,(u-\nabla q_1)\cdot\bn\,\ud y
                                        -\int_{\partial\Omega_r}\xi\,V_v\cdot\bn\,\ud y\\
&=(\nabla q_2,\nabla \xi)+\int_{\partial\Omega_r}\xi\,(u-\nabla q_1)\cdot\bn\,\ud y\\
&\qquad -\frac{1}{\pi r^2}\lt( \int_{\partial\Omega_r} q_2\,\bn\,\ud y \rt)\lt( \int_{\partial\Omega_r} \xi\,\bn\,\ud y \rt)\\
&\qquad -\frac{1}{\pi r^2}\lt(\int_{B_r}(u-\nabla q_1)\,\ud y\rt) \lt( \int_{\partial\Omega_r} \xi\,\bn\,\ud y \rt)
\end{align*}
and therefore
\begin{align}\label{eq:q2-ip}
&(\nabla q_2,\nabla\xi)
  +\frac{1}{\pi r^2}\lt( \int_{\partial\Omega_r} q_2\,\bn\,\ud y \rt)\lt( \int_{\partial\Omega_r} \xi\,\bn\,\ud y \rt)\nonumber\\
&\qquad=\int_{\partial\Omega_r}\xi(u-\nabla q_1)\cdot\bn\,\ud y
-\frac{1}{\pi r^2}\lt(\int_{B_r}(u-\nabla q_1)\,\ud y\rt) \lt( \int_{\partial\Omega_r} \xi\,\bn\,\ud y \rt) 
\end{align}
Consider $[q]=\{p\in D^{1,2}:p=q+k\;\mbox{ for some }k\in\R\}$ and 
let $\dot{D}^{1,2}$ be the the space of all equivalence classes $[q]$ with $q\in D^{1,2}$. 
Then $\dot{D}^{1,2}$ is a Hilbert space (Galdi 1994, Theorem II.5.1). We now use this fact and
the Lax-Milgram theorem to show the existence of a unique (up to a constant) solution to (\ref{eq:q2-ip}).
Let $R=1+r$ and $B_R$ an open ball of radius $R$ and centre $(0,0)$. 
For any $q\in D^{1,2}$, we have
\begin{align*}
\lt|\int_{\paro_r}q\,\bn\,\ud y\rt|
&\le |\paro_r|^{1/2}\,\|q\|_{L^2(\paro_r)}\\
&\le\,C(r)\,\|q\|_{H^1(\Omega_r\cap B_R)}\\
&\le\,C(r)\,\lt( \|q\|_{L^2(\Omega_r\cap B_R)}+\|\nabla q\|_{L^2(\Omega_r)} \rt)\\
&\le\,C(r)\,\lt( \|\nabla q\|_{L^2(\Omega_r)}
                   +\lt|\int_{B_R}q\,\ud x \rt|+\|\nabla q\|_{L^2(\Omega_r)} \rt)\\
&\le\,C(r)\,\lt( \|\nabla q\|_{L^2(\Omega_r)}+\|q\|_{L^1(B_R)} \rt)
\end{align*}
by the Sobolev trace theorem and the Poincar\'e inequality. 
Therefore the left-hand side of (\ref{eq:q2-ip}) is linear and bounded.
It is also coercive since setting $q_2=\xi$ gives an expression 
bounded below by $\|\nabla \xi\|^2$.
The right-hand side of (\ref{eq:q2-ip}) is a linear functional of $\xi$ as well. It is also bounded because $u-\nabla q_1$ is divergence-free and we can write
\begin{align*}
\lt| \int_{\partial\Omega_r}\xi(u-\nabla q_1)\cdot\bn\,\ud y \,\rt|
&\le \lt|\int_{\Omega_r}\nabla\xi \cdot (u-\nabla q_1)\,\ud y \rt|\\
&\quad\le \,\|u-\nabla q_1\|_{L^2(\R^2)}\,\|\nabla\xi\|_{L^2(\Omega_r)}.
\end{align*}
Therefore by the Lax-Milgram Theorem there exists a unique $q_2\in \dot{D}^{1,2}$
(and thus a unique up to a constant $q_2\in D^{1,2}$) that solves (\ref{eq:q2-ip}).

Having $q_1$ and $q_2$, we can find $v|_{\Omega_r}$ in terms of $u$ by seting 
$$
v=u-\nabla q_1-\nabla q_2\quad\mbox{ in } \Omega_r.
$$
It remains to uniquely define $\omega_v$ (of (\ref{eq:v_B})) and $\phi$ in terms of $u$. We let
\begin{align}\label{eq:omega}
\omega_v=\frac{c}{r^4}\int_{B_r}u\cdot y^{\bot}\,\ud y
\end{align}
(noting that $\int_{B_r}\nabla q_1\cdot y^{\bot}\,\ud y=0$),
and this then gives a unique $\phi$ by (\ref{eq:phi}).

We need to verify that $w\in G_2$ and $v\in H_r$. We have $\cdiv\phi=0$ in $B_r$ and 
$\Delta q_2=0$ in $\Omega_r$. Therefore, for any $\psi\in C_0^\infty(\R^2)$
since $\nabla q_2\cdot\bn=\phi\cdot\bn$, we can write
\begin{align*}
\int_{\R^2} w\nabla\psi\,\ud y
&=\int_{\Omega_r} w\nabla\psi\,\ud y+\int_{B_r} w\nabla\psi\,\ud y\\
&=\int_{\partial\Omega_r}\nabla q_2\cdot\bn\,\psi\,\ud y
     -\int_{\partial\Omega_r}\phi\cdot\bn\,\psi\,\ud y=0
\end{align*}
implying that $\cdiv w=0$ in the sense of distributions.

In showing $v\in H_r$ also, the only condition which is not obvious is $\cdiv v=0$ which,
having $\nabla q_2\cdot\bn=u\cdot\bn-V_v\cdot\bn-\nabla q_1\cdot \bn$, can be verified 
in a similar way to the proof of $\cdiv w=0$. \qed

\end{proof}


Before proceeding to show an appropriate bound on the $L^2$-norm 
of the difference of a divergence-free $H^1$-function and its image under 
$\OP_r$, we need to prove the following lemma which shows how the constant
of the trace inequality for the domain $\Omega_r$ depends on $r$.
We note that in this paper we use the result of this lemma only for $\Omega_r$, whose boundary is a circle, 
but the more general statement seems interesting:

\begin{lemma}\label{lem:trace}
Let $\E_r$ be an exterior domain with a bounded boundary $\partial \E_r$ of diameter $2r$ and having
the uniform $C^1$-regularity property. 

Assume that the radius of curvature of $\partial \E_r$ 
at all points is bounded below by $c_1 r$, with $c_1$ a constant independent of $r$,
and any $x\in \R^2\setminus \E_r$ is contained in an open ball of radius $c_2 r$ with $c_2\le c_1$ 
lying inside $\R^2\setminus \E_r$.

Then for any positive $\alpha<1$ there exists a constant $c$ depending only on $\alpha$ such that for all $u\in H^1(\E_r)$ 
$$
\|u\|_{L^2(\partial \E_r)}\le c\,r^{\alpha/2}\,\|u\|_{H^{1}(\E_r)}.
$$
\end{lemma}
\begin{proof}
Since $\partial \E_r$ is compact, a finite open 
cover $\{U_j\}_{j=1}^N$ of $\partial \E_r$ exists. We consider each $U_j$ to be a ball of
radius $c_2 r/2$ with its centre on $\partial \E_r$ and note that since the diameter of $\E_r$ 
and the radius of the curvature of $\partial \E_r$ are appropriately bounded, $N$ is independent of $r$. 

As in Section 5.21 and 5.22 of Adams (1975)
we define $\Psi_j$ to be the smooth transformation mapping 
$B=\{y\in\R^2:|y|<1\}$ onto $U_j$ such that 
$U_j\cap\partial \E_r=\Psi_j(B_0)$ with $B_0=\{y\in B: y_2=0\}$. Then for $U_j$ as above and
letting $x=\Psi(y)$, since the radius of curvature of $\partial \E_r$ is bounded below by $c_1 r$, we have
$$
|\frac{\partial x_i}{\partial y_k}|\le cr\quad \mbox{for}\quad i,k=1,2
$$
and
$$
|\frac{\partial (y_1,y_2)}{\partial (x_1,x_2)}|\le \frac{c}{r^2},
$$
with $c$ depending on $c_1$ and $c_2$.
For each $U_j\cap \E_r$, $j=1,\dots,N$, there exists an extension operator $E_j$ such that
for any $u\in H^1(\E_r)$,
$$
E_j[u]=u,\quad\mbox{in}\quad U_j\cap \E_r
$$
and 
$$
\|E_j[u]\|_{H^1(\R^2)}\,\le\,K_j\,\|u\|_{H^1(\E_r)}
$$
with $K_j$ independent of $r$ (Adams 1975, Theorem 4.32).
Therefore we can write
\begin{align*}
\int_{\partial \E_r}|u|^2\,\ud x
&\le\sum_{j=1}^N\int_{U_j\cap\partial \E_r}|E_j[u](x)|^2\,\ud x\\
&\le cr\sum_j\int_{-1}^1 |E_j[u](\Psi_j(y))|^2\,\ud y\\
&\le cr\sum_j
    \lt(\|E_j[u]\circ\Psi_j\|_{L^2(B)}^2+\|DE_j[u]\circ\Psi_j\|_{L^2(B)}^2 \rt)\\
&\le \sum_j
    \lt(\frac{c}{r}\|E_j[u]\|_{L^2(U_j)}^2
          +cr\|DE_j[u]\|_{L^2(U_j)}^2 \rt)\\
&\le \sum_j
    \lt(\frac{c}{r}\|E_j[u]\|_{L^{4/(1-\alpha)}(U_j)}^2\,|U_j|^{(1+\alpha)/2}
          +cr\|DE_j[u]\|_{L^2(U_j)}^2 \rt)\\
&\le \sum_j
    \lt(c\,r^{\alpha}\,\|E_j[u]\|_{H^1(\R^2)}^2+cr\|DE_j[u]\|_{L^2(\R^2)}^2\rt)\\
&\le c\,r^{\alpha}\|u\|_{H^1(\E_r)}^2
\end{align*}
noting that $N$ is independent of $r$.
\qed

\end{proof}

In what follows we will also use the Sobolev inequality
$$
\|u\|_{L^{q}(\Omega_r)}\le c\|u\|_{W^{2/p-2/q,p}(\Omega_r)}
$$
where the constant $c$ depends on $p$, $q$ and the cone $C$ determining the cone property of $\Omega_r$
(Adams 1975). Since the cone $C$ for the exterior domain $\Omega_r$ does not depend on $r$, the constant $c$ in the above inequality is independent of $r$. 
\vskip.2cm

In the last lemma of this section we show that if the vector field $u$ is divergence-free and regular enough, its projection under $\OP_r$ is arbitrarily close to $u$ in $L^2(\R^2)$, if $r$ is sufficiently small. 
The result of this lemma is essential in showing the convergence of equations
(\ref{fluid-newton})--(\ref{rigid-ic}) to Navier-Stokes equations.


\begin{lemma}\label{lem:P-Pr}
Let $u\in [H^1(\R^2)]^2$ be divergence-free and assume that $r<1$. Then for any positive $\alpha<1$ there exists a constant $c$ depending on $\alpha$ such that
$$
\|u-\OP_r u\|_{L^2(\R^2)}\,\le\, c\,r^{\alpha/2}\,\|u\|_{H^1(\R^2)}.
$$
\end{lemma}

\begin{proof}
By Lemma \ref{lem:G1G2}, since $u$ is divergence-free $\nabla q_1=0$ and  we have
$$
u=\OP_r u+\nabla q_2\quad \mbox{in }\Omega_r
$$
where $q_2$ satisfies (\ref{eq:q2}).
Letting $\xi=q_2$ in (\ref{eq:q2-ip}), we obtain
\begin{align*}
&\|\nabla q_2\|^2_{L^2(\Omega_r)}+\frac{1}{\pi r^2}\lt( \int_{\partial\Omega_r} q_2\,\bn\,\ud y \rt)^2\\
&\qquad\quad=\int_{\partial\Omega_r}q_2\, u\cdot\bn\,\ud y\nonumber
-\frac{1}{\pi r^2}\lt(\int_{B_r}u\,\ud y\rt) \lt( \int_{\partial\Omega_r} q_2\,\bn\,\ud y \rt)\\
&\qquad\quad\le \int_{\partial\Omega_r}q_2 u\cdot\bn\,\ud y+\frac{1}{2\pi r^2}\lt(\int_{B_r}u\,\ud y\rt)^2
+\frac{1}{2\pi r^2}\lt( \int_{\partial\Omega_r} q_2\,\bn\,\ud y \rt)^2
\end{align*}
and therefore
\begin{align}\label{eq:Dq2}
\|\nabla q_2\|^2_{L^2(\Omega_r)}
&+\frac{1}{2\pi r^2}\lt( \int_{\partial\Omega_r} q_2\,\bn\,\ud y \rt)^2\nonumber\\
&\le\int_{\partial\Omega_r}q_2\, u\cdot\bn\,\ud y
+\frac{1}{2\pi r^2}\lt(\int_{B_r}u\,\ud y\rt)^2.
\end{align}
Dropping the second term in the left-hand side, we can write
\begin{align*}
\|\nabla q_2\|^2_{L^2(\Omega_r)}
&\le \int_{\Omega_r}\nabla q_2\,u\,\ud y+\frac{1}{2\pi r^2}\lt(\|u\|_{L^2(B_r)}\,|B_r|^{1/2}\rt)^2\\
&\le \|\nabla q_2\|_{L^2(\Omega_r)}\,\|u\|_{L^2(\R^2)}+\frac{1}{2}\|u\|_{L^2(B_r)}^2,
\end{align*}
which gives
$$
\|\nabla q_2\|_{L^2(\Omega_r)}\,\le\,c\,\|u\|_{L^2(\R^2)}.
$$

Let 
$$
\cO_r=\{ x\in\Omega_r:\mathrm{dist}(x,\partial\Omega_r)<2 \}.
$$
If $q_2$ solves (\ref{eq:q2}), 
$$
q_2-\frac{1}{|\cO_r|}\int_{\cO_r}q_2\,\ud x
$$ 
is a solution as well 
and therefore we can assume that the average of $q_2$ is zero in $\cO_r$. 
By continuity there exists a line passing through $(0,0)$ and dividing $\cO_r$ to 
two bounded and simply connected domains  
$\cO_r^1$ and $\cO_r^2$ such that the average of $q_2$ is zero on both of them.
Noting that the boundaries of $\cO_r^1$ and $\cO_r^2$ have strong local Lipschitz property,
there exist two extension operators $E_1$ and $E_2$ such that for any $u\in H^1(\cO_r)$, and $i=1,2$ we have
$$
E_i[u]=u,\quad\mbox{in}\quad \cO_r^i,
$$
and 
$$
\|E_i[u]\|_{H^1(\R^2)}\,\le\,K_i\,\|u\|_{H^1(\cO_r^i)},
$$
with $K_i$ independent of $r$ (Adams 1975, Theorem 4.32).
Therefore by Theorem 4.4.4 of Ziemer (1989) we have
$$
\|q_2\|_{L^2(\cO_r^i)}\,\le\, C(|\cO_r^i|)\,\|\nabla q_2\|_{L^2(\cO_r^i)},\quad i=1,2,
$$
which implies that  
\begin{align*}
\|q_2\|_{L^2(\cO_r)}
&\le\,c\,\|\nabla q\|_{L^2(\Omega_r)}\\
&\le\,c\,\|u\|_{L^2(\R^2)}.
\end{align*}
Having this, we go back to (\ref{eq:Dq2}) and write
\begin{align*}
\|\nabla q_2\|^2_{L^2(\Omega_r)}
&+\frac{1}{2\pi r^2}\lt( \int_{\partial\Omega_r} q_2\,\bn\,\ud y \rt)^2\\
&\le \int_{\partial\Omega_r}q_2\, u\cdot\bn\,\ud y
     +\frac{1}{2\pi r^2}\lt(\int_{B_r}u\,\ud y\rt)^2\\
&\le \|q_2\|_{L^2(\partial\Omega_r)}\|u\|_{L^2(\partial\Omega_r)}
     +\frac{1}{2\pi r^2}\|u\|_{L^{2/(1-\alpha)}(B_r)}^2\,|B_r|^{1+\alpha}.
\end{align*}
Arguing similar to the proof of Lemma \ref{lem:trace}, we can show that
$$
\|q_2\|_{L^2(\partial\Omega_r)}\,\le\,c\, r^{\alpha/2}\,\|q_2\|_{H^1(\cO_r)}
$$
and therefore
\begin{align*}
&\|\nabla q_2\|^2_{L^2(\Omega_r)}
+\frac{1}{2\pi r^2}\lt( \int_{\partial\Omega_r} q_2\,\bn\,\ud y \rt)^2\\
&\le c\,r^{\alpha/2}\,\lt( \|q_2\|_{L^2(\cO_r)}+\|\nabla q_2\|_{L^2(\Omega_r)} \rt)\,r^{\alpha/2}\,\|u\|_{H^1(\Omega_r)}
      +c\,r^{2\alpha}\,\|u\|_{H^1(\R^2)}^2\\
&\le c\,r^{\alpha/2}\,\|u\|_{L^2(\R^2)}\,\|u\|_{H^1(\R^2)}+c\,r^{2\alpha}\,\|u\|_{H^1(\R^2)}^2.
\end{align*}
Hence, we obtain
$$
\|u-\OP_r u\|_{L^2(\Omega_r)}=\|\nabla q_2\|_{L^2(\Omega_r)}\,\le\,c\,r^{\alpha/2}\,\|u\|_{H^1(\R^2)},
$$
and
\begin{align}\label{eq:int-q2}
\frac{1}{r}\lt| \int_{\paro_r}q_2\bn\,\ud y \rt|\,\le\,c\,r^{\alpha/2}\,\|u\|_{H^1(\R^2)}.
\end{align}
%

%

To show the bound on $\|u-\OP_r u\|_{L^2(B_r)}$, we note that 
$$
\OP_r u(y)=V_v+\omega_v y^{\bot}
$$ 
for any $y\in B_r$, with $V_v\in\R^2$ and $\omega_v\in\R$
defined as in (\ref{eq:Vv}) and (\ref{eq:omega}). Therefore
\begin{align*}
\|u-\OP_r u\|_{L^2(B_r)}
&=\|u-(V_v+\omega_v y^{\bot})\|_{L^2(B_r)}\\
&\le c\,\|u\|_{L^{2/(1-\alpha)}(B_r)}|B_r|^{\alpha/2}+c\, r|V_v|+c\,r^2|\omega_v|,
\end{align*}
using H\"older inequality. From (\ref{eq:Vv}) and (\ref{eq:int-q2}) we have
\begin{align*}
r|V_v|
&\le \frac{c}{r}\,\|u\|_{L^{2/(1-\alpha)}}\,|B_r|^{(1+\alpha)/2}+\frac{c}{r}\lt| \int_{\paro_r}q_2\bn\,\ud y \rt|\\
&\le c\,r^{\alpha}\|u\|_{H^1(\R^2)}+c\,r^{\alpha/2}\,\|u\|_{H^1(\R^2)}.
\end{align*}
By (\ref{eq:omega}), we can write
$$
r^2|\omega_v|\le \frac{c}{r}\|u\|_{L^{2/(1-\alpha)}}\,r^{1+\alpha}\le c\,r^\alpha\|u\|_{H^1(\R^2)}.
$$
Therefore
\begin{align*}
\|u-\OP_r u\|_{L^2(B_r)}
\le\,c\,r^{\alpha/2}\,\|u\|_{H^1(\R^2)},
\end{align*}
and the result follows.
\qed

\end{proof}


\section{Local uniform bounds on the fluid velocity field}\label{sec:uniform}

Here under the assumption of zero angular velocity 
for the disc for $t\ge 0$, we obtain uniform in $r$ bounds 
on the velocity field.
To impose this assumption in the model (\ref{fluid-newton})--(\ref{rigid-ic}) we set  
\begin{align}\label{e:asp}
\omega_0=0\quad\mbox{ and }\quad J\eta=\int_{\partial B(t)} (x-h(t))^{\bot}\cdot\Sigma \bn\,\ud\Gamma.
\end{align}
We show that for some small enough time interval after the 
initial time, one can bound the solution of (\ref{fluid-newton})--(\ref{rigid-ic}) uniformly in $r$,
in some appropriate spaces. We then prove that these uniform bounds result in strong convergence 
of the solution in an appropriate sense.

To show the uniform bounds, we need the following estimate 
$$
\|\grad^2 u_r\|_{L^2(\Omega_r)}\le c\,\|A_r u_r\|_{L^2(\Omega_r)}
$$
with $c$ independent of $r$. We first prove this estimate.

\begin{lemma}\label{l:Galdi}
Let $u$ be the solution of
\begin{equation}\label{eq:exstokes}
\left\{
\begin{array}{ll}
\Delta u+\nabla p=f, &\mbox{ in }\R^2\setminus B_r\\
\nabla\cdot u=0,               &\mbox{ in }\R^2\setminus B_r\\
u=\dot h, &\mbox{ for }x\in\partial B_r
\end{array}
\right.
\end{equation}
with $f\in L^2(\Omega_r)$ and the constant $\dot h\in\R^2$.
Then, there exists $c>0$ independent of $r$ and $\dot h$ such that
\begin{align}\label{e:D2bnd}
\|\grad^2 u\|_{L^2(\Omega_r)}\le c\,\|f\|_{L^2(\Omega_r)}.
\end{align}
\end{lemma}

\begin{proof}
Let $y=x/r$ and $\nabla_y$ and $\Delta_y$ the gradient and Laplacian operators in $y$ coordinates. Then 
\begin{equation}\label{eq:exstokes2}
\left\{
\begin{array}{ll}
\Delta_y u+r\nabla_y p=r^2 f, &\mbox{ in }\R^2\setminus B_1\\
\nabla_y\cdot u=0,               &\mbox{ in }\R^2\setminus B_1\\
u=\dot h, &\mbox{ for }y\in\partial B_1
\end{array}
\right.
\end{equation}
Define $\Phi(x)=\dot h$ for any $x\in\R^2$ and
let $v=u-\Phi$. The vector $v$ satisfies
\begin{equation}\label{eq:exstokes2}
\left\{
\begin{array}{ll}
\Delta_y v+r\nabla_y p=r^2 f, &\mbox{ in }\R^2\setminus B_1\\
\nabla_y\cdot v=0,               &\mbox{ in }\R^2\setminus B_1\\
v=0, &\mbox{ for }y\in\partial B_1
\end{array}
\right.
\end{equation}
Now we use an estimate of Galdi (1994, equation V.4.14):
\begin{align}\label{e:D2bndy}
\|D_y^2 v\|\le cr^2\|f\|_{L^2}
\end{align}
with $c$ a constant independent of $r$. The above estimate holds if $S$, the set of solutions of (\ref{eq:exstokes2}) when $f=0$
contains only $v\equiv 0$. In our case $u\in L^2(\Omega)$ and 
$\Phi(y)\to \dot h$ as $y\to\infty$. 
Therefore $v\to -\dot h=o(\log|y|)$ and by Theorem V.3.5 of Galdi (1994)
$S=\{0\}$. Hence (\ref{e:D2bndy}) is valid. 
We now revert the coordinates to $x$ in (\ref{e:D2bndy}) and the result follows.
\qed
\end{proof}

\begin{remark}\label{r:omega}
It is for the result of the above lemma that we need to assume that
the rotation of the disc is zero. 
For a nonzero $\omega$ where $u=\dot h+x\,\omega/|x|$ on $\partial B_r$,
we have not been able to obtain a uniform bound on $\|\grad^2 u\|_{L^2(\Omega_r)}$.
\end{remark}

Now we can show the uniform estimates:

\begin{lemma}\label{lem:uniform}

Let (\ref{e:asp}) hold and consider $(u_r,h_r)$ to be a strong solution
of (\ref{fluid-newton})--(\ref{rigid-ic}) with $u_r(0)\in H^1(\Omega_r)$ and $\dot{h}_r(0)\in \R^2$. 
Then, there exist $T$ and $C$ depending on 
$\|u_r(0)\|$, $\|u_r(0)\|_{H^1}$ and $|\dot{h}_r(0)|$ but independent of $r$, such that $u_r$ is bounded by $C$ in
$$
L^2(0,T;H^2(\Omega_r))\cap L^\infty(0,T;H^1(\R^2))\cap H^1(0,T;L^2(\R^2)).
$$
\end{lemma}

\begin{proof} 
Taking the inner product of 
$$
\frac{\ud u_r}{\ud t}+A_r u_r+\OP_r\,\beta(u_r-\dot{h}_r,u_r)=0
$$
with $u_r\in H_r$, we obtain
$$
\frac{1}{2}\frac{\ud}{\ud t}\|u_r\|_{H_r}^2+\nu\|\nabla u_r\|_{L^2(\Omega_r)}^2
=-\int_{\Omega_r}((u_r-\dot{h}_r,u_r)\cdot\nabla)u_r\cdot u_r\,\ud y.
$$
Since 
$$
((u_r-\dot{h}_r)\cdot\nabla)u\cdot u=\frac{1}{2}((u_r-\dot{h}_r)\cdot\nabla)|u|^2=\frac{1}{2}\nabla\cdot(|u|^2 (u_r-\dot{h}_r)),
$$
we have
$$
\int_{\Omega_r}((u_r-\dot{h},u_r)\cdot\nabla)u_r\cdot u_r\,\ud y
=\frac{1}{2}\int_{\partial\Omega_r}|u_r|^2(u_r\cdot\bn-\dot{h}_r\cdot\bn)\,\ud y=0.
$$
Therefore 
$$
\frac{1}{2}\frac{\ud}{\ud t}\|u_r\|_{H_r}^2+\nu\|\nabla u_r\|_{L^2(\Omega_r)}^2=0
$$
which after intrgration gives the following estimates
\begin{align}\label{eq:firstbnd}
\|u_r(t)\|_{L^2}\le \|u_r(0)\|_{L^2}
\quad\mbox{and}\quad
\int_0^T \|u_r\|_{H^1(\Omega_r)}^2\,\ud t\le c\,\|u_r(0)\|_{L^2}^2
\end{align}
for almost every $t\in [0,T]$. We also note that since
$$
\dot{h}_r\cdot\bn=u_r(x)\cdot\bn\quad\mbox{for almost every }x\in\partial\Omega_r
$$
we have 
\begin{align}\label{eq:h'bnd}
|\dot{h}_r|\,\le\,c\,\|u_r\|_{W^{1,q}(\Omega_r)}.
\end{align}
with $c$ independent of $r$. To show this, for any $x\in\paro_r$, we consider the line $L_x$ tangent to $\paro_r$ 
at $x$ and dividing $\R^2$ to two parts. Using the trace theorem for $L_x$ as 
the boundary of the part that does not contain $B_r$, we can write
$$|u_r(x)|\le c\,\|u_r\|_{L^\infty(L_x)} \le c\|u_r\|_{W^{1,q}(\Omega_r)}$$
which implies the previous inequality.

Now we take the inner product of (\ref{eq:functional}) with $A_r u_r\in [L^2(\R^2)]^2$ and obtain
\begin{align*}
\frac{1}{2}\frac{\ud}{\ud t}
&\|\nabla u_r\|^2+\|A_r u_r\|^2\\
&\le \|\,|u_r|\,|\nabla u_r|\, |A_r u_r|\,\|_{L^1(\Omega_r)}
    +|\dot{h}_r|\,\|\,|\nabla u_r|\,|A_r u_r|\,\|_{L^1(\Omega_r)}\\
&\le \|u_r\|_{L^6(\Omega_r)}\,\|Du_r\|_{L^3(\Omega_r)}\,\|A_r u_r\|_{L^2(\Omega_r)}\\
&\qquad\qquad    +c\,(\|Du_r\|_{L^3(\Omega_r)}+\|u_r\|_{L^3(\Omega_r)})\,
                                 \|Du_r\|_{L^2(\Omega_r)}\,\|A_ru_r\|_{L^2(\Omega_r)}\\
&\le c\|u_r\|_{L^6(\Omega_r)}\,
         \|Du_r\|_{L^2(\Omega_r)}^{1/2}\,\|Du_r\|_{L^6(\Omega_r)}^{1/2}\,
                \|A_r u_r\|_{L^2(\Omega_r)}\\
&\qquad\qquad +c\|Du_r\|_{L^2(\Omega_r)}^{3/2}\,\|Du_r\|_{L^6(\Omega_r)}^{1/2}\,
                \|A_ru_r\|_{L^2(\Omega_r)}\\
&\qquad\qquad +c\,\|u_r\|_{L^2(\Omega_r)}^{1/2}\,\|u_r\|_{L^6(\Omega_r)}^{1/2}
                                 \|Du_r\|_{L^2(\Omega_r)}\,\|A_ru_r\|_{L^2(\Omega_r)}\\
&\le c\,\|u_r\|_{L^2(\Omega_r)}^{1/2}\,\|u_r\|_{H^1(\Omega_r)}^{1/2}\,
           \|Du_r\|_{H^1(\Omega_r)}^{1/2}\,
            \|Du_r\|_{L^2(\Omega_r)}^{1/2}\,\|A_ru_r\|_{L^2(\Omega_r)}\\
&\qquad\qquad +c\,(\|Du_r\|_{L^2(\Omega_r)}^{1/2}+\|A_ru_r\|_{L^2(\Omega_r)}^{1/2})\,
            \|Du_r\|_{L^2(\Omega_r)}^{3/2}\,\|A_ru_r\|_{L^2(\Omega_r)}
\end{align*}
for almost every $t\in [0,T]$. We therefore conclude that
\begin{align}
\frac{\ud}{\ud t}
\|\nabla u_r\|^2+\|A_r u_r\|^2
&\le c(1+\|u_r(0)\|_{L^2(\R^2)}^4)(1+\|\grad u_r\|_{L^2(\Omega_r)}^2)^3\label{eq:Du}.
\end{align}
Dropping $\|A_r u_r\|^2_{L^2(\Omega_r)}$ and integrating, we obtain
$$
1+\|\grad u_r(t)\|^2_{L^2(\Omega_r)}
\le \frac{1+\|u_r(0)\|_{H^1}^2}
{\sqrt{1-ct(1+\|u_r(0)\|_{L^2(\R^2)}^4)\,(1+\|u_r(0)\|^2_{H^1})^2}}.
$$
Therefore for almost every $t\le T$ with
\begin{align}\label{eq:T1}
T=\frac{1}{2c(1+\|u_r(0)\|_{L^2(\R^2)}^4)\,(1+\|u_r(0)\|^2_{H^1})^2},
\end{align}
we have 
$$
\|u_r(t)\|_{H^1(\Omega)}\le C(\|u_r(0)\|,\|u_r(0)\|_{H^1},|h'_r(0)|,|\omega_0|).
$$
By Theorem \ref{thm:taktuc}, and also since by (\ref{eq:firstbnd}), $M|\dot{h}|^2$ and $J|\omega|^2$ 
are \ uniformly bounded by $\|u_r(0)\|_{H_r}$, we conclude that $u_r$ is$\ $  uniformly bounded in $L^\infty (0,T;H^1(\R^2))$.
With the above bound, we integrate (\ref{eq:Du}) and obtain a uniform bound 
on $\|A_r u_r\|_{L^2(0,T;L^2(\Omega_r))}$. 
This, by Lemma \ref{l:Galdi} and the uniform bound on $\|u_r(t)\|_{H^1(\Omega)}$, implies that 
$$
\|u_r(t)\|_{H^2(\Omega)}\le C(\|u_r(0)\|,\|u_r(0)\|_{H^1},|h'_r(0)|,|\omega_0|).
$$
Using (\ref{eq:h'bnd}), we have
\begin{align}\label{eq:h'bndi}
\int_0^{T}|\dot{h}|^4\,\ud t
\,\le\, c\,\sup_{t\in [0,T]}\|u_r(t)\|_{H^1(\Omega_r)}^2 \int_0^T \|u_r\|_{H^2(\Omega_r)}^2\,\le\,C.
\end{align}
Finally since $\dot{u}_r=-A_ru_r-\OP_r\beta(u_r-\dot{h}_r,u_r)$, we have
\begin{align*}
\|\dot{u}\|_{L^2(\R^2)}
&\le c\,\|A_ru_r\|_{L^2(\R^2)}+c\,\|((u_r-\dot{h}_r)\cdot\nabla)u_r\|_{L^2(\Omega_r)}\\
&\le c\,\|u_r\|_{H^2(\Omega_r)}
    +c\,\|u_r(t)\|_{L^2(\Omega_r)}^{1/2}\,\|u_r(t)\|_{H^1(\Omega_r)}\,\|u_r(t)\|_{H^2(\Omega_r)}^{1/2}\\
&\qquad\quad +c\,|\dot{h}_r|\,\|u_r(t)\|_{H^1(\Omega_r)}
\end{align*}
which gives a uniform bound on $\|\dot{u}_r\|_{L^2(0,T;L^2(\R^2))}$. \qed
\end{proof}
%


\begin{lemma}\label{lem:conv}
Let (\ref{e:asp}) hold and consider $(u_r,h_r)$ to be a strong solution of (\ref{fluid-newton})--(\ref{rigid-ic}) 
with $u_r$ and $\dot{h}_r$ uniformly bounded in 
$$
L^2(0,T;H^2(\Omega_r))\cap L^\infty(0,T;H^1(\R^2))\cap H^1(0,T;L^2(\R^2))
\,\mbox{ and }\, L^4(0,T;\R^2)
$$ 
respectively. 
Then 
there exists 
$$
u\in L^2(0,T;H^2(\R^2)\cap L^\infty(0,T;H^1(\R^2) \cap H^1(0,T;L^2(\R^2))
$$ 
and $\dot{h}\in L^4(0,T;\R^2)$ 
and a real positive decreasing sequence $\{r_n\}_{n\in\N}$ converging to zero, 
such that, 
as $n\to \infty$, $u_{r_n}$ converges strongly 
to $u$ in $L^2(0,T;H^1(\Omega_s))$ for any $s>0$,$\ $ and 
$\;\dot{h}_{r_n}$ converges strongly to $\dot{h}$ in $L^2(0,T;\R^2)$.
\end{lemma}

\begin{proof}
By Lemma \ref{lem:uniform}, $u\in L^\infty(0,T;H^1(\R^2) \cap H^1(0,T;L^2(\R^2))$. 
To show that $u\in L^2(0,T;H^2(\R^2))$, we note that
since $u_r$ is uniformly bounded 
in $L^2 (0,T;H^2(\Omega_{r}))$, for any $s>r$ we have
\begin{align}\label{eq:pre1}
\|u_{r}\|_{ L^2 (0,T;H^2(\Omega_{s}))}\le C
\end{align}
with $C$ independent of $r$ and $s$. Therefore for any $s>0$, there exists a sequence $\{u_{r_n}\}_{n\in\N}$
which converges weakly to $u\in L^2(0,T;H^2(\Omega_s))$ as $n\to\infty$. Hence, by uniqueness of weak limits
$u\in L^2(0,T;H^2(\Omega_s))$ for any $s>0$. 
Now we note that by Theorem 2.1.4 of Ziemer (1989)
a function $w\in L^2(\Omega)$ with $\Omega\subset \R^2$, is in $H^1(\Omega)$ if and only if $w$ is absolutely continuous on almost all lines in $\Omega$ parallel to the coordinate axes, and its partial derivatives belong to $L^2(\Omega)$. Consider a real positive sequence $\{s_n\}_{n\in\N}$ converging to zero as $n\to\infty$. 
Since $u\in L^2(0,T;H^2(\Omega_{s_n}))$ for any $n$, $Du$ is absolutely continuous on all lines parallel to coordinates axes in $\Omega_{s_n}$ and the partial derivatives of all elements of $Du$ are uniformly (in $n$) 
bounded in $L^2(\Omega_{s_n})$. 
Then, since 
$\cap_{n=1}^{\infty} \R^2\setminus\Omega_{s_n}=\{(0,0)\}$, 
and noting that the countable intersection of sets of full measure in $[0,T]$, 
has full measure in $[0,T]$, it follows that 
$Du\in L^2(0,T;H^1(\R^2))$. This then, since $u\in L^2(0,T;H^1(\R^2))$, 
implies that $u\in L^2(0,T;H^2(\R^2))$. 

We now show the strong convergence in $L^2(0,T;H^1(\Omega_s))$ for any $s>0$. Since $u_r$ is 
uniformly bounded in $L^2(0,T;V_r)\cap L^\infty(0,T;H_r)\cap H^1(0,T;V_r^*)$, it is uniformly bounded
in $L^2(0,T;V_s)\cap L^\infty(0,T;H_s)\cap H^1(0,T;V_s^*)$ for any fixed $s>r$ and therefore 
there exists a subsequence $\{u_{r_n}\}_{n\in\N}$ 
(with $r_n<r_m$ if $n>m$) which converges strongly in $L^2(0,T;H_s)$ as $n\to \infty$ 
(see Robinson 2001 or Temam 1977 for example). 
Therefore for any $\epsilon>0$ there exists some $N$ such that for $n>N$
\begin{align}\label{eq:pre2}
\|u-u_{r_n}\|_{L^2(0,T;L^2({\R^2}))}\le \epsilon.
\end{align}
For any fixed $s>0$, let $E_s[u-u_{r_n}]$ be an extension of $u-u_{r_n}$ from $\Omega_{s}$ to $\R^2$, such that for almost every $t\in[0,T]$
\begin{align}\label{eq:E}
E_s[u-u_{r_n}]=u-u_{r_n}\quad\mbox{almost everywhere in }\Omega_{s},
\end{align}
and
\begin{align}\label{eq:E-bnd}
\|E_s[u-u_{r_n}]\|_{H^k(\R^2)}\le c_k\|u-u_{r_n}\|_{H^k(\Omega_s)}.
\end{align}
with $k\ge 0$ and $c_k=c_k(s)$ (one can show that $c_0$ is independent of $s$
but for our purpose here it makes no difference). 
By Ehrling's lemma for any $\eta>0$ there exists a $c_\eta>0$ such that 
\begin{align*}
&\|E_s[u-u_{r_n}]\|_{L^2(0,T;H^1(\R^2))}\\
&\qquad\quad\le \eta \|E_s[u-u_{r_n}]\|_{L^2(0,T;H^2(\R^2))}
    + c_\eta \|E_s[u-u_{r_n}]\|_{L^2(0,T;L^2(\R^2))}.
\end{align*}
Let $\epsilon>0$ and take $\eta=\epsilon/c_2(s)$. For any $n\ge m$, we can write
\begin{align*}
&\|u-u_{r_n}\|_{L^2(0,T;H^1(\Omega_{s}))}
\le\|E_s[u-u_{r_n}]\|_{L^2(0,T;H^1(\R^2))}\\
&\qquad\le \eta \|E_s[u-u_{r_n}]\|_{L^2(0,T;H^2(\R^2))}
    + c_\eta \|E_s[u-u_{r_n}]\|_{L^2(0,T;L^2(\R^2))}\\ 
&\qquad\le \epsilon \|u-u_{r_n}\|_{L^2(0,T;H^2(\Omega_{s}))}
    + c_0\,c_\eta \|u-u_{r_n}\|_{L^2(0,T;L^2(\Omega_{s}))}.
\end{align*}
By ({\ref{eq:pre2}}) for $n$ big enough we have
$$
\|u-u_{r_n}\|_{L^2(0,T;L^2(\Omega_s))}\le \frac{\epsilon}{c_0\,c_\eta}.
$$
Also by (\ref{eq:pre1})
$$
\|u-u_{r_n}\|_{L^2(0,T;H^2(\Omega_s))}\le C
$$
with $C$ independent of $n$ and $s$. Hence, for any $\epsilon>0$ there exists some $N_1$ depending on $\epsilon$ and $s$ such that 
$$
\|u-u_{r_n}\|_{L^2(0,T;H^1(\Omega_{s}))}<C\epsilon
$$
for $n>N_1$. Therefore for a fixed $s$, 
$\|u-u_{r_n}\|_{L^2(0,T;H^1(\Omega_{s}))}\to 0$ as $n\to\infty$.

It remains to show that $\dot{h}_r\to \dot{h}$ in $L^2(0,T;\R^2)$. For any $n>m$ we have
\begin{align*}
|(\dot{h}_{r_m}&-\dot{h}_{r_n})\cdot \bn|\\
&=\lt|\big(u_{r_m}(r_m\bn)-u_{r_n}(r_n\bn)\big)\cdot\bn\rt|\\
&=\lt|\big(u_{r_m}(r_m\bn)-u_{r_n}(r_m\bn)\big)\cdot\bn+\big( u_{r_n}(r_m\bn)-u_{r_n}(r_n\bn)\big)\cdot\bn\rt|\\
&\le \sup_{x\in\partial\Omega_{r_m}}|u_{r_m}(x)-u_{r_n}(x)|+c\,\|u_{r_n}\|_{H^{3/2}(\Omega_{r_n})}\,{r_m}^{\alpha/2}\\
&\le c\,\|u_{r_m}-u_{r_n}\|_{H^1(\Omega_{r_m})}^{1/2}\,\|u_{r_m}-u_{r_n}\|_{H^2(\Omega_{r_m})}^{1/2}\\
&\qquad   +c\,\|u_{r_n}\|_{H^1(\Omega_{r_n})}^{1/2}\,\|u_{r_n}\|_{H^2(\Omega_{r_n})}^{1/2}\,{r_m}^{\alpha/2}
\end{align*}
with $\alpha<1$. Since the above inequality is true for any $\bn$, we conclude that for $n>m>N$,
$$
\int_0^T |\dot{h}_{r_m}-\dot{h}_{r_n}|^2\,\ud t\to 0 \quad\mbox{as}\quad N \to \infty
$$
and obtain the required strong convergence. \qed
\end{proof}

The above lemma shows the strong convergence of subsequences of $u_r$ and $h_r$. 
However, after showing that the limit $u$ satisfies the Navier-Stokes equations in the 
whole plane and therefore is unique, one can show the above strong convergence result for any $r$ rather 
than only a subsequence $\{r_n\}$: we argue along the lines of the proof of Lemma 3.1 of
Robinson (2004). We suppose that for some $s>0$, $u_r$ does not converge to $u$ (strongly) in 
$L^2(0,T;H^1(\Omega_s))$. This means that there exists a subsequence $\{u_{r_m}\}$ and a $\delta>0$
such that
$$
\|u_{r_m}-u\|_{L^2(0,T;H^1(\Omega_s))}>\delta.
$$
But $u_{r_m}$, by Lemma \ref{lem:uniform} is uniformly bounded in 
$$
L^2(0,T;H^2(\Omega_r))\cap L^\infty(0,T;H^1(\R^2))\cap H^1(0,T;L^2(\R^2))
$$
and therefore by Lemma \ref{lem:conv}, has a subsequence that converges to $u$
strongly in $L^2(0,T;H^1(\Omega_s))$, contradicting the above inequality.
One can argue similarly and conclude that $\dot h_r\to \dot{h}$ in $L^2(0,T;\R^2)$. 
Therefore, from now on, we consider 
$u_r$ (rather than $u_{r_n}$) converging strongly to $u$ in $L^2(0,T;H^1(\Omega_s))$
for any $s>0$,
and $\dot{h}_r$ (rather$\,$ than $\,\dot{h}_{r_n}$) converging strongly to $\dot{h}$ in $L^2(0,T;\R^2)$.


\begin{corollary}
Let $u_r$, $h_r$, $u$ and $h$ be as in Lemma \ref{lem:conv}. Then
$$
\|u-u_r\|_{L^2(0,T;H^1(\Omega_r))}\to 0,\quad\mbox{as }r\to 0.
$$
\end{corollary}

\begin{proof}

Let $0<r<s$. We can write
\begin{align*}
\int_0^T\|u&-u_{r}\|^2_{H^1(\Omega_{r})}\,\ud t\\
&=\int_0^T\|u-u_{r}\|^2_{H^1(\Omega_{s})}+\|u-u_{r}\|^2_{H^1(\Omega_{r}\setminus\Omega_{s})}\,\ud t\\
&\le \int_0^T\|u-u_{r}\|^2_{H^1(\Omega_{s})}+c\,s\,\|u-u_{r}\|^2_{W^{1,4}(\Omega_{r}\setminus\Omega_{s})}\,\ud t\\
&\le \|u-u_{r}\|^2_{L^2(0,T;H^1(\Omega_{s}))}+c\,s\,\|u-u_{r}\|^2_{L^2(0,T;H^2(\Omega_{r}))}
\end{align*}
using H\"older inequality in the second line. 
Since $\|u-u_{r}\|^2_{L^2(0,T;H^2(\Omega_{r}))}\le C$ with $C$ independent of $r$, 
for any $\epsilon>0$ there exists a $\delta$ (independent of $r$) such that $s<\delta$ is small enough to ensure that
$$
c\,s\,\|u-u_{r}\|^2_{L^2(0,T;H^2(\Omega_{r}))}<\epsilon/2.
$$
Also, by Lemma \ref{lem:conv} for $r$ small enough
$$
\|u-u_{r}\|^2_{L^2(0,T;H^1(\Omega_{s}))}<\epsilon/2.
$$
The result therefore follows.\qed

\end{proof}


\section{The equations at the zero limit of the rigid body radius}\label{sec:limiteqn}

Now we can prove the main result of this paper:

\begin{theorem} \label{thm:limeqn}

Let (\ref{e:asp}) hold and consider $(u_r,h_r)$ to be a solution of (\ref{fluid-newton})-(\ref{rigid-ic}) 
with the initial condition $u_r(0)\in H^1(\R^2)$ satisfying 
$u_r(0)=\dot{h}_r(0)$ 
for $y\in\paro_r$ and $\dot{h}_r(0)\in\R^2$. 
Suppose also that there exists $u_0\in H^1(\R^2)$ such that $u_r(0)\rightharpoonup u_0$ in $L^2(\R^2)$ as $r\to 0$ and $\|u_r(0)\|_{H^1(\Omega_r)}\le c\|u_0\|_{H^1(\R^2)}$ for any $r>0$. 
Then
\begin{itemize}
\item[i)] $u$, the limit of $u_r$ at the zero limit of $r$ satisfies the Navier-Stokes equations in the whole plane with initial condition $u_0$,

\item[ii)] the trajectory of the centre of the rigid disc converges to a fluid particle trajectory as $r\to 0$.
\end{itemize}
\end{theorem}

\begin{proof}
{\it i}) We first show the result for $t\in [0,T]$, with $T$ as in (\ref{eq:T1}), and then use
the properties of the solution of the Navier-Stokes equations in the whole of $\R^2$ to
get the result for any $0\le t<\infty$. In what follows 
$$H=\{u\in [L^2(\R^2)]^2: \cdiv u=0\}$$
and
$$V=\{u\in [H^1(\R^2)]^2: \cdiv u=0\}.$$
$V^*$ is the dual of $V$,
and 
\begin{align*}
(u,v)_{\Omega_r}=\int_{\Omega_r}u\cdot v\,\ud y.\\
\nonumber
\end{align*}
%

%
1) $\int_0^T\int_{\Omega_r} (A_ru_r-Au)\cdot v\,\ud y\,\ud t \to 0\;$ as $r\to 0$ for any $v\in L^2(0,T;H)$:\\

We show the weak-* convergence of $A_r u_r$ to $Au$ in $L^2(0,T,V^*)$ 
and then use the fact that $\|u_r\|_{L^2(0,T;H^2(\Omega_r))}$ 
and $\|u\|_{L^2(0,T;H^2(\Omega_r))}$ are bounded independent of $r$,
to conclude the required convergence for and $v\in L^2(0,T;H)$.

Let $v\in L^2(0,T;V)$ and $v_r={\OP}_rv$ and write
\begin{align}\label{Au-weak*}
\int_0^T (A_r u_r-Au,v)\,\ud t
&= \int_0^T(\cA_r u_r,v_r)\,\ud t+\int_0^T (\nu\Delta u,v)\,\ud t\nonumber\\
&= \int_0^T(\mathcal{A}_ru_r+\nu\Delta u,v)\,\ud t
+\int_0^T(\mathcal{A}_r u_r,v_r-v)\,\ud t.
\end{align}
We have
\begin{align*}
\int_0^T (\mathcal{A}_r u_r,v_r-v)\,\ud t
&\le\; c\,\|u_r\|_{L^2(0,T;H^2(\Omega_r))}\|v_r-v\|_{L^2(0,T;L^2(\Omega_r))}\\
&\;+ \int_0^T\frac{c}{r^2}\int_{\partial\Omega_r}D(u_r){\bf n}\,\ud\Gamma
            \cdot\int_{B_r}(v_r-v)\,\ud y\,\ud t\\
&\;+ \int_0^T\frac{c}{r^3}
         \int_{\partial\Omega_r}{\bf n}^{\bot}\cdot D(u_r){\bf n}\,\ud \Gamma
         \int_{B_r}(v_r-v)\cdot y^{\bot}\,\ud y\,\ud t.
\end{align*}
By characterization of $G_2$ in Lemma \ref{lem:G1G2}, the last term is zero. 
For the second term in the right-hand side of the above inequality, 
by Lemma \ref{lem:trace} and \ref{lem:P-Pr} we can write
\begin{align*}
\frac{c}{r^2}\int_{\paro_r}&D(u_r){\bf n}\,\ud y
        \cdot\int_{B_r}(v_r-v)\,\ud y\\
&=\frac{c}{r^2}\int_{\paro_r}D(u_r){\bf n}\,\ud y
        \cdot\lt(-\int_{B_r}(v_r-v)\,\ud y\rt)\\
&\le\, 
       \frac{c}{r^2}\,\|D(u_r)\|_{L^2(\paro_r)}\,(2\pi r)^{1/2}\,
       \|v-v_r\|_{L^2(B_r)}\,(\pi r^2)^{1/2}\\
&\le\,
       {c}{r^{\alpha-1/2}}\,\|u_r\|_{H^2(\Omega_r)}
                      \,\|v\|_{H^1(\R^2)}.
\end{align*}
for any $\alpha<1$, and therefore 
\begin{align*}
\int_0^T(\mathcal{A}_r u_r,v_r-v)\,\ud t\to 0\quad\mbox{as}\quad r\to 0.
\end{align*}
For the integrand in the first term of the last line of (\ref{Au-weak*}) we can write
\begin{align*}
(\mathcal{A}_r u_r+\nu\Delta u,v)=
&-2\nu\int_{\Omega_r} (D(u_r)-D(u)):D(v)\,\ud y\\
&+2\nu\int_{\paro_r} (D(u_r)-D(u))\,v\cdot{\bf n}\,\ud \Gamma
     +\nu\int_{B_r} \Delta u\cdot v\,\ud y\\
&-\frac{2\nu}{M}\int_{B_r}v\,\ud y\cdot
          \int_{\paro_r}D(u_r)\bn\,\ud\Gamma\\
&-\frac{2\nu}{J}r \int_{B_r}v\cdot y^{\bot}\,\ud y 
         \int_{\partial B_r}{\bf n}^{\bot}\cdot D(u_r){\bf n}\,\ud\Gamma.
\end{align*}
Since with $0<\alpha<1$,
\begin{align*}
\frac{2}{M}\int_{B_r}v\,\ud y\,\cdot
&\int_{\paro_r}|D(u_r)\,\,\ud\Gamma\\
&\le\, 
     \frac{c}{r^2}\|v\|_{L^6(B_r)}{(\pi r^2)^{5/6}}\,\|Du_r\|_{L^2(\partial B_r)}\,\sqrt{2\pi r}\\
&\le\, c\,r^{\alpha/2+1/6}\,\|v\|_{H^1(\R^2)}\,\|u_r\|_{H^2(\Omega_r)} ,
\end{align*}
\begin{align*}
\frac{2}{J}r \int_{B_r}v\cdot y^{\bot}\,\ud y
&\int_{\partial B_r}{\bf n}^{\bot}\cdot D(u_r){\bf n}\,\ud\Gamma\\
&\le\, \frac{c}{r^3}\|v\|_{L^4}\,(r^4)^{3/4}
               \,\|Du_r\|_{L^2(\partial B_r)}\sqrt{2\pi r}\\
&\le\, c\,r^{\alpha/2+1/2}\,\|v\|_{H^1(\R^2)}\,\|u_r\|_{H^2(\Omega_r)} ,
\end{align*}
and
\begin{align*}
\int_{B_r} \Delta u\cdot v\,\ud y
&\le\, \|u\|_{H^2(\R^2)}\,\|v\|_{L^4(B_r)}|B_r|^{1/4}\\
&\le\, c\,r^{1/2}\,\|v\|_{H^1(\R^2)}\,\|u_r\|_{H^2(\Omega_r)},
\end{align*}
we have
\begin{align*}
\int_0^T (\mathcal{A}_r u_r+\nu\Delta u,v)\,\ud t
&\,\le\, c\,\|u-u_r\|_{L^2(0,T;H^1(\R^2))}\,\|v\|_{L^2(0,T;H^1(\R^2))}\\
&\,+c\,r^{\alpha}\|u_r-u\|_{L^2(0,T;H^2(\R^2))}\,\|v\|_{L^2(0,T;H^1(\R^2))}\\
&\;+c\,r^{1/2}\,\|u_r\|_{L^2(0,T;H^2(\R^2))}\,\|v\|_{L^2(0,T;H^1(\R^2))}
\end{align*}
which converges to zero as $r$ tends to zero. We conclude that 
$A_r u_r-Au\overset{*}{\rightharpoonup}0$ in $L^2(0,T;V^*)$.

Now to prove the required convergence for $v\in L^2(0,T;H)$, we note that since $L^2(0,T;V)$ is dense in $L^2(0,T;H)$, for any $w\in L^2(0,T;H)$, there exists a sequence $\{v_m\}_{m=1}^\infty$ in $L^2(0,T;V)$, such that $v_m\to w$ in $L^2(0,T;H)$ as $m\to \infty$. We can write
\begin{align*}
&\lt|\int_0^T(A_ru_r-Au,w)_{\Omega_r}\,\ud t\,\rt|\\
&\quad\le \lt|\int_0^T(A_r u_r-Au,w-v_m)_{\Omega_r}\,\ud t\,\rt|
   +\lt|\int_0^T(A_ru_r-Au,v_m)_{\Omega_r}\,\ud t\,\rt|\\
&\quad\le \|A_r u_r-Au\|_{L^2(0,T;L^2(\Omega_r))}\|w-v_m\|_{L^2(0,T;L^2(\Omega_r))}\\
&\hspace{2in}+\lt|\int_0^T(A_r u_r-Au,v_m)_{\Omega_r}\,\ud t\,\rt|.
\end{align*}
For any $\epsilon>0$, there exists $M$ big enough such that $\|w-v_m\|\le \epsilon/(2C)$ for $m>M$ and with 
$C\ge \|A_r u_r-Au\|_{L^2(0,T;H(\Omega_r))}$. Also by the weak-* convergence we just showed, 
there exists some $r=r(M)$ such that 
$$
\lt|\int_0^T(A_r u_r-Au,v_m)_{\Omega_r}\,\ud t\,\rt|<\epsilon/2,\quad \mbox{for }r<r(M).
$$
Therefore for any $\epsilon>0$, 
$$
\lt|\int_0^T(A_r u_r-Au,w)_{\Omega_r}\,\ud t\,\rt|<\epsilon
$$
when $r$ is small enough.
\\

%
2) $\OP_r[\beta(u_r-\dot{h}_r,u_r)]\rightharpoonup \OP[((u-\dot{h})\cdot\nabla)u]$ in $L^2(0,T;H)$:\\

We note that $\OP[((u-\dot{h})\cdot\nabla)u]$ is bounded in $L^2(0,T;H)$, since
\begin{align}\label{eq:Buu}
\|\OP[((u-\dot{h})\,\cdot &\nabla)u]\|^2_{L^2(0,T;H)}\nonumber\\
&\le\, \|(u\cdot \nabla)u\|^2_{L^2(0,T;H)}+\|(\dot{h}\cdot \nabla)u\|^2_{L^2(0,T;H)}\nonumber\\
&\le\, \|u\|_{L^\infty(0,T;H^1(\R^2))}\,\|u\|_{L^2(0,T;H^2(\R^2))}\nonumber\\
&\qquad\quad +\|\dot{h}\|_{L^4(0,T;\R^2)}\,\|u\|_{L^4(0,T;H^1(\R^2))}.
\end{align}

We first show the weak-* convergence for $L^{4/3}(0,T;V^*)$. Let $v\in L^4(0,T;H^1)$ and $v_r=\OP_r v$ as before and write
\begin{align*}
\int_0^T 
&\big(\OP_r[\beta(u_r-\dot{h}_r,u_r)]- \OP[((u-\dot{h})\cdot\nabla)u],v\big)\,\ud t\\
&=\int_0^T \big([(u_r-\dot{h}_r)\cdot\nabla]u_r,v_r\big)_{\Omega_r}\,\ud t
                                      +\int_0^T \big([(u-\dot{h})\cdot\nabla]u,v\big)\,\ud t\\
&=\int_0^T \big([(u_r-\dot{h}_r)\cdot\nabla]u_r-[(u-h')\cdot\nabla]u,v\big)_{\Omega_r}\,\ud t\\
&+\int_0^T \big([(u_r-\dot{h}_r)\cdot\nabla]u_r,v_r-v\big)_{\Omega_r}\,\ud t
      +\int_0^T \big([(u-\dot{h})\cdot\nabla]u,v\big)_{B_r}\,\ud t.
\end{align*}
We have
\begin{align*}
&\int_0^T\big( ((u_r-\dot{h}_r)\cdot\nabla)u_r,v_r-v \big)_{\Omega_r}\,\ud t\\
&\qquad\,\le\, \Big( \|u_r\|_{L^\infty(0,T;H^1(\Omega_r))}\|u_r\|_{L^2(0,T;H^2(\Omega_r))}\\
&\qquad\qquad +\|\dot{h}\|_{L^2(0,T;\R^2)}\|u_r\|_{L^\infty(0,T;H^1(\Omega_r))} \Big)\|v-v_r\|_{L^2(0,T;L^2(\Omega_r))},
\end{align*}
\begin{align*}
\int_0^T\big( &(\dot{h}_r\cdot \nabla)u_r - (\dot{h}\cdot\nabla)u,v\big)_{\Omega_r}\,\ud t\\
&=\int_0^T\big( ((\dot{h}_r-\dot{h})\cdot \nabla)u_r,v \big)_{\Omega_r}
              + \big( (\dot{h}\cdot\nabla)(u_r-u),v \big)_{\Omega_r}\,\ud t\\
&\le \|\dot{h}_r-\dot{h}\|_{L^2(0,T;\R^2)}\,\|u_r\|_{L^\infty(0,T;H^1(\Omega_r))}\,\|v\|_{L^2(0,T;L^2(\R^2))}\\
&+\|\dot{h}\|_{L^4(0,T;\R^2)}\,\|u_r-u\|_{L^2(0,T;H^1(\Omega_r))}\,\|v\|_{L^4(0,T;L^2(\R^2))},
\end{align*}
with $w=u_r-u$
\begin{align*}
\int_0^T \big( &(u_r\cdot\nabla)u_r-(u\cdot\nabla)u,v \big)_{\Omega_r}\,\ud t\\
&=\int_0^T \big( (w\cdot \nabla)w),v\big)_{\Omega_r}
                +\big( (w\cdot \nabla)u),v\big)_{\Omega_r}
                +\big( (u\cdot \nabla)w),v\big)_{\Omega_r}\,\ud t\\
&\le c\,\Big( \|w\|_{L^\infty(0,T;H^1(\Omega_r))}+\|u\|_{L^\infty(0,T;H^1(\Omega_r))} \Big)\,\\
&\qquad\qquad\qquad \times \|w\|_{L^2(0,T;H^1(\Omega_r))}\,\|v\|_{L^2(0,T;H^1(\Omega_r))},
\end{align*}
and 
\begin{align*}
\int_0^T &\big([(u-\dot{h})\cdot\nabla]u,v\big)_{B_r}\,\ud t\\
&\le \int_0^T c\,\Big( \|u\|_{H^1(\R^2)}\,\|u\|_{H^2(\R^2)}+|\dot{h}|\,\|u\|_{H^1(B_r)} \Big)\,\|v\|_{L^4(B_r)}\,|B_r|^{1/4}\\
&\le c\,r^{1/2}\,\Big( \|\dot{h}\|_{L^2(0,T;\R^2)}
             +\|u\|_{L^2(0,T;H^2(\R^2))} \Big)\,\\
&\qquad\qquad\qquad \times\|u\|_{L^\infty(0,T;H^1(\R^2))}\,\|v\|_{L^2(0,T;H^1(\R^2))}
\end{align*}
We therefore conclude the required weak-* convergence. The weak convergence in $L^2(0,T;H)$ then follows 
in a similar way to the previous case, since both $\OP_r[\beta(u_r-\dot{h}_r,u_r)]$ and $\OP[((u-\dot{h})\cdot\nabla)u]$ 
are bounded in $L^2(0,T;H)$ and $L^4(0,T;V)$ is dense in $L^2(0,T;H)$.
\\

%
3) $\int_0^T\big( \ud (u_r-u)/\ud t,v \big)_{\Omega_r}\,\ud t\to 0\;$ 
as $r\to 0$ for any $v\in L^2(0,T;H)$:\\

As before, we first show the convergence for more regular $v\in L^2(0,T;V)$.
For $s>r$ we can write
\begin{align*}
&\int_0^T \lt( \frac{\ud u_r}{\ud t}-\frac{\ud u}{\ud t},v \rt)_{\Omega_r}\,\ud t\\
&\;=\int_0^T \lt( \frac{\ud u_r}{\ud t}-\frac{\ud u}{\ud t},v \rt)_{\Omega_s}\,\ud t
+\int_0^T \lt( \frac{\ud u_r}{\ud t}-\frac{\ud u}{\ud t},v \rt)_{\Omega_r\setminus\Omega_s}\,\ud t.
\end{align*}
For $v\in L^2(0,T;V)$ we have
\begin{align*}
&\int_0^T \lt( \frac{\ud u_r}{\ud t}-\frac{\ud u}{\ud t},v \rt)_{\Omega_r\setminus\Omega_s}\,\ud t\\
&\quad\le \int_0^T \lt\| \frac{\ud u_r}{\ud t}-\frac{\ud u}{\ud t} \rt\|_{L^2({\Omega_r\setminus\Omega_s})}
             |{\Omega_r\setminus\Omega_s}|^{1/4}\,\|v\|_{L^4({\Omega_r\setminus\Omega_s})}\\
&\quad \le c\, s^{1/2} \|v\|_{L^2(0,T;V)} 
\end{align*}
since $\|\dot{u}_r\|_{L^2(0,T;L^2(\Omega_r))}$ is uniformly bounded by Lemma \ref{lem:uniform}
and $\|\dot{u}\|_{L^2(0,T;L^2(\R^2))}$ is bounded by Lemma \ref{lem:conv}. 
Therefore, for any $\epsilon$ there exists $\delta>0$ such that for any $r$ and $s$ less than $\delta$
$$
\int_0^T \lt( \frac{\ud u_r}{\ud t}-\frac{\ud u}{\ud t},v \rt)_{\Omega_r\setminus\Omega_s}\,\ud t<\epsilon/2.
$$ 
By Lemma \ref{lem:conv}, $\dot{u}_r\rightharpoonup \dot{u}$ 
as $r\to 0$ in $L^2(0,T;L^2(\Omega_s))$ for any fixed $s$ and hence there exists $r^*>0$ such that
$$
\int_0^T \lt( \frac{\ud u_r}{\ud t}-\frac{\ud u}{\ud t},v \rt)_{\Omega_s}\,\ud t\le \epsilon/2
$$
for $r<r^*$. Let $r=min\{\delta/2,r^*\}$. Then
\begin{align*}
\int_0^T \lt( \frac{\ud u_r}{\ud t}-\frac{\ud u}{\ud t},v \rt)_{\Omega_r}\,\ud t\le \epsilon
\end{align*}
and we get the required convergence for $v\in L^2(0,T;V)$.

The converegence for $v\in L^2(0,T;H)$ follows in a similar way to the parts (1) and (2), since 
$\|\dot{u}_r\|_{L^2(0,T;L^2(\Omega_r))}$ and $\|\dot{u}\|_{L^2(0,T;L^2(\R^2))}$ are both bounded by a constant independent of $r$.\\


The above convergence results and (\ref{eq:functional}) imply that $u(t)$ satisfies
\begin{align*}
\int_0^T\int_{\Omega_r}\lt(\frac{\ud u}{\ud t}+Au+\OP[((u-\dot{h})\cdot\nabla)u]\rt)
\cdot v \,\ud y\,\ud t\to 0\quad\mbox{as }\,r\to 0.
\end{align*}
for any $v\in L^2(0,T;H)$. Since $\ \dot{u},Au\in L^2(0,T;H)$ by Lemma \ref{lem:conv}$\ $ 
and $\OP[((u-\dot{h})\cdot\nabla)u]\in L^2(0,T;H)$ by (\ref{eq:Buu}), in the zero limit of $r$, we have
\begin{align}\label{eq:nseh}
\frac{\ud u}{\ud t}+Au+\OP[((u-\dot{h})\cdot\nabla)u]=0\quad\mbox{as an equality in }L^2(0,T;H).\\
\nonumber
\end{align}

We need to show the limiting velocity field $u$ has initial condition $u_0$.
Let $\phi\in C^1(0,T;H)$ with $\phi(T)=0$. Taking the inner product 
of the above equation and (\ref{eq:functional}) with $\phi(t)$ we obtain
$$
\int_0^T (u,\dot\phi)_{\Omega_r}+( Au+\OP[((u-\dot{h})\cdot\nabla)u],\phi )_{\Omega_r}
\,\ud t=(u(0),\phi(0))_{\Omega_r}
$$
and
$$
\int_0^T (u_r,\dot\phi)_{\Omega_r}+( A_ru_r+\OP_r[\beta(u-\dot{h},u)],\phi )_{\Omega_r}\,\ud t=(u_r(0),\phi(0))_{\Omega_r}
$$
respectively. Comparing these two at the zero limit of $r$ implies that $u(0)=u_0$.\\

Since $h\in H^1(0,T;\R^2)\subset L^\infty(0,T;\R^2)$, we can perform the following change of coordinates in (\ref{eq:nseh})
\begin{align*}
x(t)=y(t)+h(t).
\end{align*}
Letting
$$
U(x,t)=u(x+h(t),t)
$$
we can write (\ref{eq:nseh}) as
\begin{align}\label{eq:nsel}
\frac{\ud U}{\ud t}+AU+\OP[(U\cdot\nabla)U]=0\quad\mbox{as an equality in }L^2(0,T;H).\\
\nonumber
\end{align}
%
%

Now, we extend the result to all times:\\

Taking the inner product of the above equation with $AU$, using appropriate 
estimates for the nonlinear term and then integrating, we obtain
\begin{align}\label{eq:H1u}
\|u\|_{H^1(\R^2)}\le \frac{C_1}{2}=c\,\e^{c\|u_0\|_{L^2}}\,\|u_0\|_{H^1(\R^2)}.
\end{align}
Let $\|u_r(0)\|_{H^1(\Omega_r)}=C_1$. By Lemma {\ref{lem:uniform}}, we have
$\|u_r(t)\|_{H^1(\Omega_r)}<2C_1$ for 
$$
t<T^*=\frac{1}{2c(1+\|u_r(0)\|_{H_r}^4)\,(1+C_1^2)^2}<T
$$
(noting that $\|u_r(t)\|_{H_r}\le \|u_r(0)\|_{H_r}$ for any $t>0$). Since 
$$
\int_0^T\|u_r-u\|^2_{H^1(\Omega_r)}\to 0\quad \mbox{as } r\to 0, 
$$
there exists $r^*$ small enough such that 
for some $t>T-T^*/2$ 
$$
\|u_{r^*}(t)-u(t)\|_{H^1(\Omega_{r^*})}<C_1/2
$$
and therefore $\|u_{r^*}(t)\|_{H^1(\Omega_{r^*})}\le C_1$ and by above for $t<T+T^*/2$, $\|u_r(t)\|_{H^1(\Omega_r)}$ is uniformly bounded.
Hence, by the first part of the proof $u$ satisfies (\ref{eq:nseh}) for $t\in[0,T+T^*/2]$. 
The same estimate (\ref{eq:H1u}) holds for $0\le t<T+T^*/2$ as well. 
Therefore the same argument as above
gives the result of the theorem for $0\le t<T+T^*$ and continuing in a similar way for $0\le t<T+nT^*/2$ 
for any $n\in\mathbb{N}$.\\

{\it ii}) To show the convergence of the trajectory of the centre of the tracer, we fix a point $y^*$ on $\paro_r$ and write
\begin{align*}
|u((0,0),t)-\dot{h}(t)|
&\le |u((0,0),t)-u(y^*,t)|\\
&\quad+|u(y^*,t)-\dot{h}_r(t)|+|\dot{h}_r(t)-\dot{h}(t)|.
\end{align*}
We have
$$
|u((0,0),t)-u(y^*,t)|\le c\,r^{\alpha/2}\,\|u\|_{H^{3/2}(\R^2)}
$$
for $\alpha<1$ by Lemma 7.26 of Renardy \& Rogers (2004).
To find an appropriate bound on the second term in the right-hand side of the above inequality, we note that
for any $\bn=y/|y|$ with $y\in \paro_r$ we can write
\begin{align*}
|\bn\cdot(u(y^*,t)&-\dot{h}_r(t))|\\
&=|\bn\cdot(u(y^*,t)-u_r(y,t))|\\
&\le |u(y^*,t)-u(y,t)|+|u(y,t)-u_r(y,t)|\\
&\le c\,r^{\alpha/2}\,\|u\|_{H^{3/2}(\R^2)}+c\,\|u-u_r\|_{H^1(\Omega_r)}^{1/2}\,\|u-u_r\|_{H^2(\Omega_r)}^{1/2}.
\end{align*}
Since the above inequality is true for all $\bn$ as above, we conclude that
\begin{align*}
|u(y^*,t)-\dot{h}_r(t)|
&\le c\,r^{\alpha/2}\,\|u\|_{H^{3/2}(\R^2)}\\
&\quad+c\,\|u-u_r\|_{H^1(\Omega_r)}^{1/2}\,\|u-u_r\|_{H^2(\Omega_r)}^{1/2}.
\end{align*}
Therefore we have
\begin{align*}
\int_0^T 
&|u((0,0),t)-\dot{h}(t)|^2\,\ud t\\
&\le c\,r^{\alpha}\,\|u\|_{L^\infty(0,T;H^1(\R^2))}\,\|u\|_{L^2(0,T;H^2(\R^2))}\\
&+c\,\|u-u_r\|_{L^2(0,T;H^2(\Omega_r))}\,\|u-u_r\|_{L^2(0,T;H^1(\Omega_r))}+\|\dot{h}_r-\dot{h}\|_{L^2(0,T;\R^2)}.
\end{align*}
The right-hand side converges to zero as $r\to 0$ by Lemma \ref{lem:conv}. Hence
\begin{align}\label{eq:hU}
\int_0^T |\dot{h}(t)-U(h(t),t)|\,\ud t=0,
\end{align}
implying that $h(t)$ is the trajectory of the fluid particle initially at $(0,0)$.
Since 
\begin{align*}
|h(t)-h_r(t)|
\le\|h-h_r\|_{H^1(0,T;\R^2)},
\end{align*}
we conclude that $h_r(t)\to h(t)$ as $r\to 0$ by Lemma (\ref{lem:conv}) . 
\qed
\end{proof}

We note that there exist $u_0$ and $u_r(0)$ satisfying the assumption of
Theorem {\ref{thm:limeqn}}. For instance we can consider $u_0\in H^1(\R^2)$ such that $u_0=\dot{h}$ for $y\in B_{R}$ with
$R>0$ and some $\dot{h}\in\R^2$, and let $u_r(0)=u_0$ for any $r<R$.  

\section{Conclusion}

We studied the limiting motion 
of a system of a rigid disc moving with a fluid flow in $\R^2$ as
the radius of the disc goes to zero.
We showed that if the disc is not allowed to rotate,
the trajectory of the centre of the disc converges to a fluid particle trajectory.

Two related problems are the motion of the fluid-rigid body in a two-dimensional 
bounded domain, and the case of more than one rigid disc. 
In the case of a bounded domain, the existence of a global strong solution 
of a fluid-rigid body system is shown by Takahashi (2003) assuming that the rigid body does not touch the boundary.
With the same assumption, it seems that a similar approach as the one presented here
(but more technical since the change of coordinates would be more complicated)
can give the same convergence result. 
When there are several rigid bodies moving with the fluid, which is the case in real experiments, 
it is known that at least one weak solution exists when the radius of rigid bodies are fixed 
(Desjardins \& Esteban 1999; Feireisl 2002; Grandmont \& Maday 2000; San Martin et al 2002), but to our knowledge,
the uniqueness of the solution is not known. 

The analysis here goes some way towards justifying 
the use of tracer particles for finding Lagrangian paths of fluid flow, 
but we have unfortunately had to restrict 
to the case in which the particles are not allowed to rotate. 
Obtaining appropriate uniform bounds in the case that the particle can rotate
seems more 
challenging (see Remark \ref{r:omega}), and presents a very interesting open problem.

\begin{acknowledgement}

We are grateful to Yun Wang for pointing out a problem with 
the bound on $\|D^2 u\|_{L^2}$ in the original version of the paper, and to Igor Kukavika who suggested 
a way round this problem by assuming that the disc does not rotate.
MD was supported by a Warwick Postgraduate Research Fellowship, and as a postdoc by EPSRC and ERC.
JCR is supported by an EPSRC
Leadership Fellowship EP/G007470/1.

\end{acknowledgement}

\section*{References}{\footnotesize
\begin{description}

\item\textsc{Adams, R.A.} 1975, {\it Sobolev Spaces}. Academic Press, New York.

\item\textsc{Conca, C.; San Martin, J.; Tucsnak, M.} 2000, {\it Existence of solutions for the equations modelling the motion of a rigid body in a viscous fluid.} Comm. Partial Differential Equations, 25, 1019-1042.

\item\textsc{Desjardins, B.; Esteban, M. J.} 1999, {\it Existence of weak solutions for the motion of rigid bodies in a viscous fluid.} Arch. Ration. Mech. Anal.,  146, 59--71

\item\textsc{Desjardins, B.; Esteban, M. J.} 2000, {\it On weak solutions for fluid-rigid structure interaction: compressible and incompressible models.}  Comm. Partial Differential Equations  25,  no. 7-8, 1399--1413.

\item\textsc{Feireisl, E.} 2002, 
{\it On the motion of rigid bodies in a viscous fluid.} 
Mathematical theory in fluid mechanics (Paseky, 2001). Appl. Math.  47, 463--484.

\item\textsc{Galdi, G. P.} 1994, {\it An introduction to the mathematical theory of the Navier-Stokes equations. Vol. I. Linearized steady problems}. Springer Tracts in Natural Philosophy, 38. Springer-Verlag, New York.

\item\textsc{Galdi, G. P.; Silvestre, A. L.} 2002,
{\it Strong solutions to the problem of motion of a rigid body in a Navier-Stokes liquid under the action of prescribed forces and torques.} 
Nonlinear problems in mathematical physics and related topics, I, 121--144,
Int. Math. Ser. (N. Y.), 1, Kluwer/Plenum, New York. 

\item\textsc{Grandmont, C.; Maday, Y.} 2000,
{\it Existence for an unsteady fluid-structure interaction problem.}
M2AN Math. Model. Numer. Anal. 34, no. 3, 609--636. 

\item\textsc{Gunzburger, M. D.; Lee, H.-C.; Seregin, G. A.} 2000,
{\it Global existence of weak solutions for viscous incompressible flows around a moving rigid body in three dimensions.}
J. Math. Fluid Mech. 2, no. 3, 219--266. 

\item\textsc{Hoffmann, K.-H.; Starovoitov, V. N.} 1999,
{\it On a motion of a solid body in a viscous fluid. Two-dimensional case.}
Adv. Math. Sci. Appl. 9, no. 2, 633--648. 

\item\textsc{Iftimie, D.; Lopes Filho, M. C.; Nussenzveig Lopes, H. J.} 2006,
{\it Two-dimensional incompressible viscous flow around a small obstacle.}
Math. Ann. 336, no. 2, 449--489. 

\item\textsc{Judakov, N. V.} 1974,
{\it The solvability of the problem of the motion of a rigid body in a viscous incompressible fluid.} 
Dinamika Splo\v sn. Sredy Vyp. 18 Dinamika Zidkost. so Svobod. Granicami , 249--253, 255.

\item\textsc{Renardy, M.; Rogers, R. C.} 2004
{\it An introduction to partial differential equations.}
Second edition. Texts in Applied Mathematics, 13. Springer-Verlag, New York,.

\item\textsc{Robinson, J.C.} 2001, {\it Infinite-dimensional dynamical systems.} Cambridge Texts in Applied
Mathematics, Cambridge University Press, Cambridge.

\item\textsc{Robinson, J.C.} 2004, {\it A coupled particle-continuum model: well-posedness and the limit of zero radius}
 Proc. R. Soc. Lond. Ser. A Math. Phys. Eng. Sci.  460,  no. 2045, 1311--1334. 

\item\textsc{San Martín, J. A.; Starovoitov, V.; Tucsnak, M.} 2002
{\it Global weak solutions for the two-dimensional motion of several rigid bodies in an incompressible viscous fluid.}
Arch. Ration. Mech. Anal. 161, no. 2, 113--147. 

\item\textsc{Serre, D.} 1987, {\it Chute libre d'un solide dans un fluide visqueux incompressible. Existence}
 Japan J. Appl. Math.  4 ,  no. 1, 99--110. 

\item\textsc{Takahashi, T.} 2003, {\it Existence of strong solutions for the problem of a rigid-fluid system}, C.R. Acad. Sci. Paris, Ser. I 336, 453-458.

\item\textsc{Takahashi, T.; Tucsnak, M.} 2004, {\it Global strong solutions for the two-dimen\-sional motion of an infinite cylinder in a viscous fluid}, J. Math. Fluid Mech., 6, 53-77.

\item\textsc{Temam, R.} 1977, {\it Navier-Stokes Equations}. AMS Chelsea Publishing, Providence, RI.

\item\textsc{Ziemer, W. P.} 1989,
{\it Weakly differentiable functions.
Sobolev spaces and functions of bounded variation}. Graduate Texts in Mathematics, 120. Springer-Verlag, New York.

\end{description}
}

\address{Mathematics Institute \\University of Warwick \\Coventry CV4 7AL\\ UK.\\
 {\tt m.dashti@warwick.ac.uk} \\{\tt j.c.robinson@warwick.ac.uk} }

\end{document}